\documentclass[a4paper]{amsart}

\usepackage{amsmath,amsthm}
\usepackage{mathrsfs}
\usepackage{amssymb}
\usepackage{enumitem}
\usepackage[usenames,dvipsnames]{pstricks}
\usepackage{epsfig}
\usepackage{esint}    
\usepackage{pinlabel}
\usepackage{graphicx}
\usepackage{cancel}

\newtheorem{theor}{Theorem}[section]
\newtheorem{cor}[theor]{Corollary}
\newtheorem{lemma}[theor]{Lemma}
\newtheorem{prop}[theor]{Proposition}

\newtheorem{question}{Question}
\newtheorem*{question*}{Question}

\theoremstyle{definition}
\newtheorem{defn}{Definition}

\theoremstyle{remark}

\newtheorem{remark}[theor]{Remark}

\numberwithin{equation}{section}
\numberwithin{defn}{section}

\newcommand{\N}{\mathbb{N}}        
\newcommand{\R}{\mathbb{R}}        
\renewcommand{\S}{\mathbb{S}}        

\renewcommand{\H}{\mathcal{H}}

\newcommand{\Hut}{H^1_{\mathrm{tan}}(N, \, \mathbb{S}^2)}
\newcommand{\Hu}{H^1(N, \, \R^3)}

\DeclareMathOperator{\infess}{ess\,inf} 
 
\DeclareMathOperator{\Id}{Id}

\DeclareMathOperator{\dist}{dist}

\DeclareMathOperator{\VMO}{VMO}
\DeclareMathOperator{\BMO}{BMO}
\DeclareMathOperator{\ind}{ind}
\DeclareMathOperator{\proj}{proj}

\renewcommand{\d}{\mathrm{d}}

\newcommand{\abs}[1]{\left| #1 \right|}
\newcommand{\norm}[1]{\left\| #1 \right\|}

\newcommand{\csubset}{\subset\!\subset}

\usepackage{color}
\usepackage[normalem]{ulem} 

\title[Morse's index formula in VMO for manifolds with boundary]
{Morse's index formula in VMO\\ for compact manifolds with boundary}

\author{Giacomo Canevari}
\address{Sorbonne Universit\'es,
        UPMC --- Universit\'e Paris 06,
        CNRS, UMR 7598, Laboratoire Jacques-Louis Lions,
        4 place Jussieu, \mbox{75005} Paris, France.}
        \curraddr{Mathematical Institute,
        University of Oxford,
        Andrew Wiles Building,
        Radcliffe Observatory Quarter,
        Woodstock Road,
        \mbox{OX2 6GG} Oxford, United Kingdom.}
\email[G. Canevari]{canevari@maths.ox.ac.uk}
  
\author{Antonio Segatti}
\author{Marco Veneroni}
\address{Dipartimento di Matematica ``F. Casorati'',
  Universit\`a di Pavia,
  Via Ferrata 1, \mbox{27100} Pavia, Italy.}
\email[A. Segatti]{antonio.segatti@unipv.it}
\email[M. Veneroni]{marco.veneroni@unipv.it}

\date{August 20, 2015} 

\begin{document}

\begin{abstract}
In this paper, we study Vanishing Mean Oscillation vector fields on a compact manifold with boundary. 
Inspired by the work of Brezis and Nirenberg, we construct a topological invariant --- the index --- for such fields, and establish the analogue of Morse's formula.
As a consequence, we characterize the set of boundary data which can be extended to nowhere vanishing $\VMO$ vector fields.
Finally, we show briefly how these ideas can be applied to (unoriented) line fields with $\VMO$ regularity,  thus providing a reasonable framework for modelling a surface coated with a thin film of nematic liquid crystals.
\end{abstract}



\maketitle

\section{Introduction}
\label{sec:intro}

The starting point of the investigations developed in this paper is the
analysis of a variational model for nematic shells. 
Nematic shells are the datum of a two-dimensional surface
(for simplicity, at a first step, without boundary) 
$N\subset\mathbb{R}^3$
coated with a thin film of nematic liquid crystal 
(\cite{KRV11, LubPro92,
NapVer12L,NapVer12E,Nelson02,
ssvARMA,ssvPRE,Straley71}).
This line of research has attracted a lot of attention from the physics
community due to its vast technological applications (see \cite{Nelson02}).
From the mathematical point of view, nematic shells offer an interesting
and nontrivial interplay between calculus of variations, partial differential 
equations, geometry and topology. 
The basic mathematical description of nematic shells
consists in an energy defined on tangent vector fields with unit length, named directors.
This energy, in the simplest situation, takes the form
 \begin{equation}
\label{eq:energy}
\mathcal{E}(n) :=\int_{N}\vert \nabla n\vert^2 \d S,
\end{equation}
where $\nabla$ stands for the covariant derivative of the surface $N$. 
If one is interested in the minimization of this energy, the 
first step is to understand whether there are competitors
for the minimization process. For this type of energy, 
the natural functional space where to look for minimizers
is the space of tangent vector fields with $H^1$ regularity.
This means,
recalling that we are looking for vector fields with unit norm,
the space defined in this way
\begin{equation}
\label{eq:Hut}
\Hut :=\left\{n \in \Hu\colon n(x)\in T_x(N) \hbox{ and } 
\vert n\vert =1 \hbox{ a.e.}  \right\}.
\end{equation}
Now, the problem turns into the 
understanding of the topological conditions on $N$,
if any, that make $\Hut$ empty or not.
Note that this problem, in the case $N=\mathbb{S}^2$,
is indeed a Sobolev version
of the celebrated hairy ball problem
concerning the existence of
a tangent
vector field with unit norm on the two-dimensional sphere.
The answer, when dealing with continuous fields, is
negative. This is a consequence of a more general result,
the Poincar\'e-Hopf Theorem, that relates
the existence of a smooth tangent vector field with unit norm to 
the topology of $N$. More precisely, a smooth vector
field with unit norm exists if and only if $\chi(N)=0$, where
$\chi$ is the Euler characteristic of $N$. 
In case $N$ is a compact surface in $\R^3$, the Euler
 characteristic can be written as a function of the topological
 genus $k$:
 \[
  \chi(N) = 2(1 - k) .
 \] 
In \cite{ssvARMA} it has been proved, using calculus
of variations tools, that the very same result holds for 
vector fields with $H^1$ regularity.
Therefore, up to diffeomorphisms, the only compact surface
in~$\R^3$ which admits a unit norm vector field in 
$H^1$ is the torus, corresponding to $k = 1$.
On the other hand, it is easy to comb the sphere with
a field $v\in W^{1, p}_{\mathrm{tan}}(\S^2, \, \S^2)$ for all $1 \leq p < 2$. 
It is interesting to note that this result could be seen
as a ``non flat'' version of a well know result of Bethuel
that gives conditions for the non emptiness of the space
$$
H^{1}_{g}(\Omega, \, \mathbb{S}^1):=\left\{ v\in H^{1}(\Omega, \, \R^2)\colon
\vert v(x)\vert =1 \hbox{ a.e. in } \Omega \,\,\hbox{ and } v\equiv g \hbox{ on } 
\partial \Omega
\right\},
$$
where $\Omega$ is a simply connected bounded domain in $\R^2$ and $g$ 
is a prescribed
smooth boundary datum with $\vert g\vert =1$. 
The non-emptiness of $H^{1}_{g}(\Omega, \, \mathbb{S}^1)$ is related
to a topological condition on the Dirichlet datum $g$ (see \cite{Bethuel96}
and \cite{BBB})
while in the result in \cite{ssvARMA} the topological constraint is on the genus of the surface.

Instead of using the standard Sobolev theory, we reformulate this
problem in the space of Vanishing Mean Oscillation (VMO) 
functions, introduced by Sarason in \cite{Sarason},  which 
constitute a special subclass of Bounded Mean Oscillations functions, defined by John and Nirenberg in \cite{JN}.
We recall the definitions and some properties of these objects in Section~\ref{sect: VMO}, 
but we immediately note that
VMO contains the critical spaces with respect to Sobolev embeddings, that is,
\begin{equation}
\label{eq:embedding}
 W^{s, p}(\R^n) \subset \VMO(\R^n) \qquad \textrm{when } sp = n , \ 1 < s < n .
\end{equation}
In a sense, $\VMO$ functions are a good surrogate for the continuous
 functions, because some classical topological constructions can
  be extended, in a natural way, to the $\VMO$ setting.
In particular, we recall here the $\VMO$ degree theory, which has been
 developed after Brezis and Nirenberg's seminal papers \cite{BN1}
  and \cite{BN2}.

Besides relaxing the regularity
on the vector field, we will consider $n$-dimensional compact and connected 
submanifolds of $\R^{n+1}$ and,
instead of fixing the length of the vector field to be $1$, 
we will look for vector fields which are bounded and
uniformly positive. 

Thus, the problem of combing a two-dimensional surface
with $H^1$ vector fields can be generalized in the following way.

\begin{question} \label{quest: comb N}
 Let $N$ be a compact, connected submanifold of $\R^{n + 1}$, without boundary, of dimension $n$. Does a vector field $v\in\VMO(N, \, \R^{n + 1})$, satisfying
 \begin{equation} \label{v}
  v(x)\in T_x N \quad \textrm{and}  \quad c_1 \leq \abs{v(x)} \leq c_2 
 \end{equation}
  for a.e. $x\in N$ and some constants $c_1, \, c_2 > 0$, exists?
\end{question}
 
 The first outcome of this paper is to provide a complete answer to Question~\ref{quest: comb N}. By means of the Brezis and Nirenberg's degree theory, we can show that the
 existence of nonvanishing vector fields in $\VMO$ is subject to the same topological obstruction as in the continuous case, that is, we prove the following
 
 \begin{prop} \label{prop: unbounded}
  Let $N$ be a compact, connected $n$-dimensional submanifold of $\R^{n + 1}$, without boundary. There exists a function $v\in\VMO(N, \, \R^{n + 1})$ satisfying~\eqref{v} if and only if $\chi(N) = 0$.
 \end{prop}
 
 \medskip

After addressing manifolds without boundary, we consider the case where $N$ is a manifold with boundary, and we prescribe Dirichlet boundary conditions to the vector field $v$ on $N$.
The main issue of this paper is to understand which are the 
topological conditions on the manifold~$N$ and on the Dirichlet boundary datum
that guarantee the existence of a nonvanishing and bounded tangent vector field on $N$
extending the boundary condition. Applications of these results can be found in variational problems
for vector fields that satisfy a prescribed boundary condition of Dirichlet type, e.g., in the framework of liquid crystal shells.

More precisely, we address the following problem:
 
\begin{question} \label{quest: comb boundary}
 Let $N \subset \R^d$ be a compact, connected and orientable $n$-submanifold with boundary. Let \mbox{$g\colon \partial N \to \R^d$} be a boundary datum in $\VMO$, satisfying 
 \begin{equation} \label{hp g VMO}
  g(x)\in T_x N \quad \textrm{and} \quad c_1 \leq \abs{g(x)} \leq c_2 
 \end{equation}
 for $\H^{n - 1}$-a.e. $x\in \partial N$ and some constants $c_1, \, c_2 > 0$. Does a field $v\in \VMO(N, \, \R^d)$, which fulfills~\eqref{v} and has trace $g$ (in some sense, to be specified), exist?
\end{question}

When working in the continuous setting, a similar issue can be investigated with the help of a topological tool: the index of a vector field. In particular, even in this weak framework, 
we expect conditions that 
relate the index of the boundary conditions with the index of the tangent vector field
and the Euler characteristic of $N$.
In order to understand the difficulties and to ease the presentation, we recall here some definitions related to the degree theory and an important property.

First, we recall Brouwer's definition of degree.
 Let $N$ be as in Question~\ref{quest: comb boundary} and let $M$ be a connected, orientable manifold without boundary, of the same dimension as $N$.
Let $\varphi\colon N \to M$ be a smooth map, and let $p\in M\setminus \varphi(\partial N)$ be a regular value for $\varphi$ (that is, the Jacobian matrix $D \varphi(x)$ is non-singular for all $x\in \varphi^{-1}(p)$).
We define the degree of $\varphi$ with respect to $p$ as
\[
 \deg(\varphi, \, N, \, p) := \sum_{x\in \varphi^{-1}(p)} \mathrm{sign }(\det D \varphi(x)).
\]
This sum is finite, because $\varphi^{-1}(p)$ is a discrete set (as $\varphi$ is locally invertible around each point of $\varphi^{-1}(p)$) and $N$ is compact. 

It can be proved that, if $p_1$ and $p_2$ are two regular values in the same component of \mbox{$M \setminus\varphi(\partial N)$}, then $\deg(\varphi, \, N, \, p_1) = \deg(\varphi, \, N, \, p_2)$.
Since the regular values of $\varphi$ are dense in $M$ (by Sard lemma), the definition of $\deg(\varphi, \, N, \, p)$ can be extended to every \mbox{$p\in M \setminus \varphi(\partial N)$}. 
Moreover, by approximation it is possible to define the degree when $\varphi$ is just continuous.
In case $N$ is a manifold without boundary, $\deg(\varphi, \, N, \, p)$ does not depend on the choice of $p\in M$, so we will denote it by $\deg(\varphi, \, N, \, M)$.
Let us mention also that, if $N$ and $M$ are compact and without boundary, the following formula holds:
\begin{equation} \label{deg_formula}
 \deg(\varphi, \, N, \, M) = \frac{1}{\tau(M)}\int_N \varphi^*(\d \tau) = \frac{1}{\tau(M)}\int_N \det D\varphi(x) \, \d\sigma(x) ,
\end{equation}
where $\sigma, \, \tau$ are the Riemannian metrics on $N$ and $M$, respectively.

Ideally, given a continuous vector field $v$, one would like to define its index by
\[
 \ind(v, \, N) = \deg(v, \, N, \, 0) .
\]
However, this is not possible, because in order to define the degree it is essential that the domain and the target manifold have the same dimension.
This is not the case here, since the domain manifold $N \subset \R^d$ has dimension strictly less that the target manifold $\R^d$.
To overcome this issue, there are at least two different strategies.
The one we consider in this paper, which is also the most widely studied in the literature
(see, e.g., \cite{Morse, GuiPol74, LLoyd, Milnor, Spivak}),
is to use coordinate charts to represent $v$, locally around its zeros, as a map $\R^n \to \R^n$.
This requires an additional assumption, namely that the zero set of $v$ is discrete.
Thus, within this approach, an approximation technique is needed in order to extend the definition of index to any continuous field.
This construction,  based on the Transversality Theorem, is explained in detail in Section~\ref{app: cont index}.
Another possibility is to consider an open neighbourhood $U \subset \R^d$ of $N$, and extend $v$ to a map $w\colon U \to \R^d$, in a suitable way. 
Then, it would make sense to write
\[
 \ind(v, \, N) := \deg(w, \, U, \, 0) ,
\]
and this would give an equivalent definition of the index.
This approach is inspired by a classical proof of the Poincar\'e-Hopf theorem, which can be found in \cite[Theorem 1, p. 38]{Milnor}).
Some details of this construction are given in Remark~\ref{remark: tubular}.

Once the index has been properly defined, it can been used to establish a precise relation between the behaviour of a vector field $v$ and the topological properties of $N$.
Denote by $\partial_- N$ the subset of the boundary where $v$ points inward 
(that is, letting $\nu(x)$ be the outward unit normal to $\partial N$ in $T_x N$, we have $x\in \partial_- N$ if and only if $v(x)\cdot \nu(x) < 0$).
Call $P_{\partial N}v$ the vector field on $\partial N$ defined by
\[
 P_{\partial N} v(x) := \proj_{T_x \partial N} v(x) \qquad \textrm{for all } x\in \partial N .
\]
Morse proved the following equality (see \cite{Morse}), which was later rediscovered and generalized by Pugh (see \cite{Pugh}) and Gottlieb (see \cite{Gott, GottS}).
 
 \begin{prop}[Morse's index formula] \label{prop: morse index}
 If $v$ is a \emph{continuous} vector field over $N$ satisfying $0\notin v(\partial N)$, with finitely many zeros, and if $P_{\partial N} v$ has finitely many zeros, then
 \begin{equation}
 \label{eq:morse}
  	\ind(v, \, N) + \ind(P_{\partial N}v, \, \partial_- N) = \chi(N),
 \end{equation}
 where $\chi(N)$ is the Euler characteristic of $N$.
 \end{prop}

In figure~\ref{fig:one} we plot some examples on $N=\overline{B_r(0)}$. In this case $\chi(N)=1$.

\begin{figure}[h]
\centering 
\parbox{.3\textwidth}{
	\centering
	\includegraphics[height=4cm]{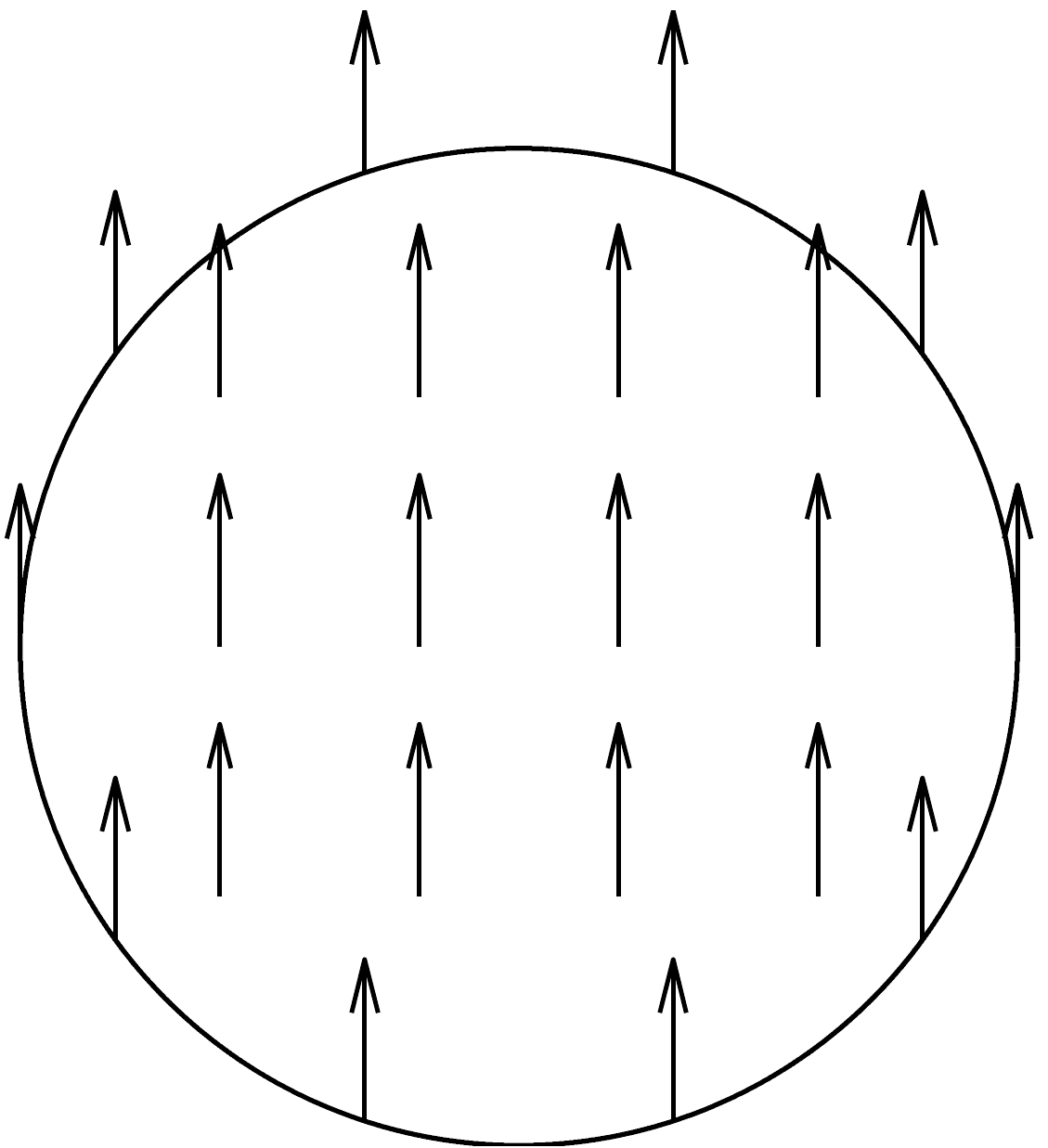}
	\mbox{a) $v_1(x,y)=(0,1);$}}
	\quad 
\begin{minipage}{.3\textwidth}
	\centering
	\includegraphics[height=4cm]{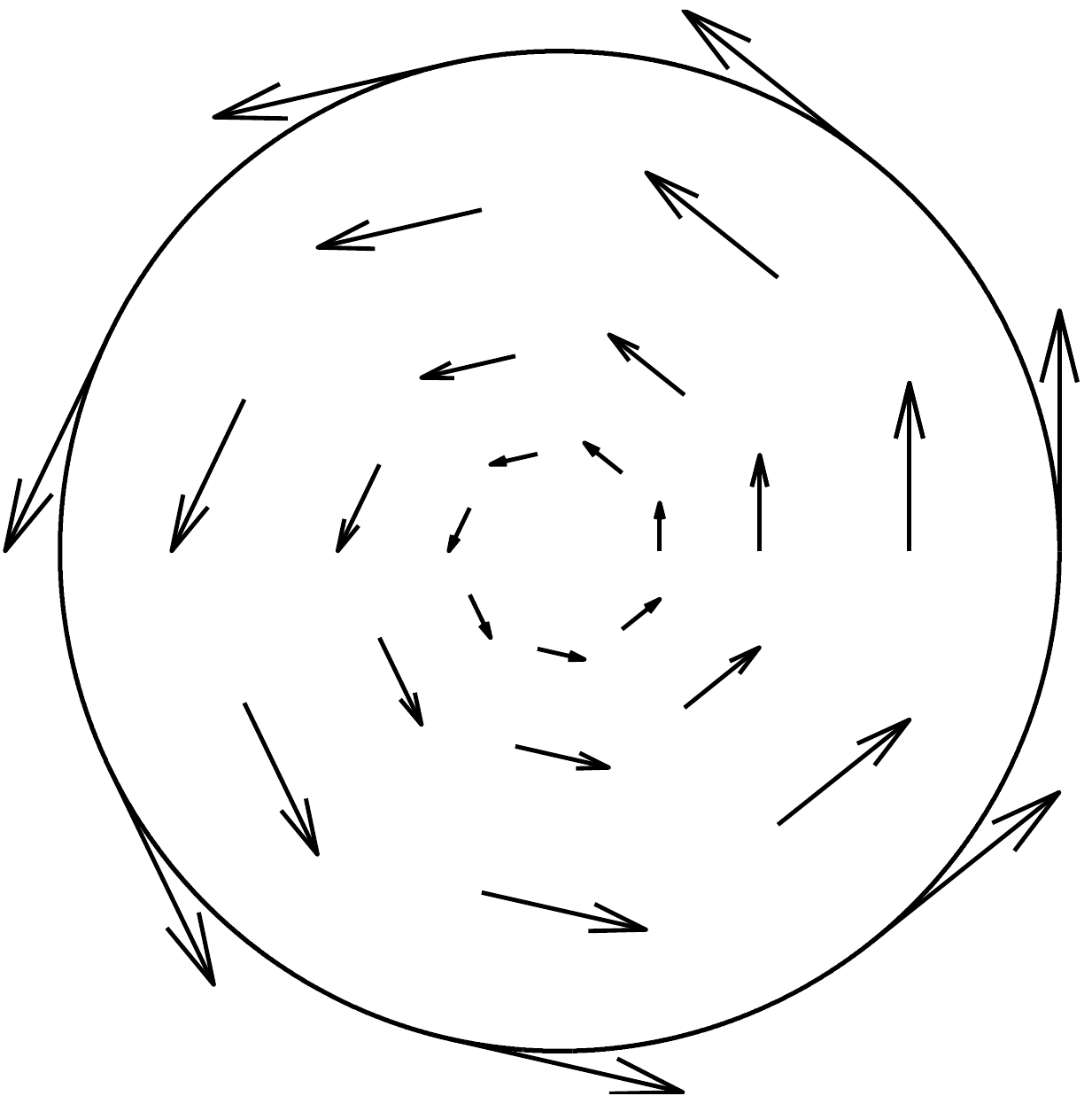}
	\mbox{b) $v_2(x,y)=(-y,x);$}
\end{minipage}
	\qquad 
\begin{minipage}{.3\textwidth}
	\centering
	\includegraphics[height=4cm]{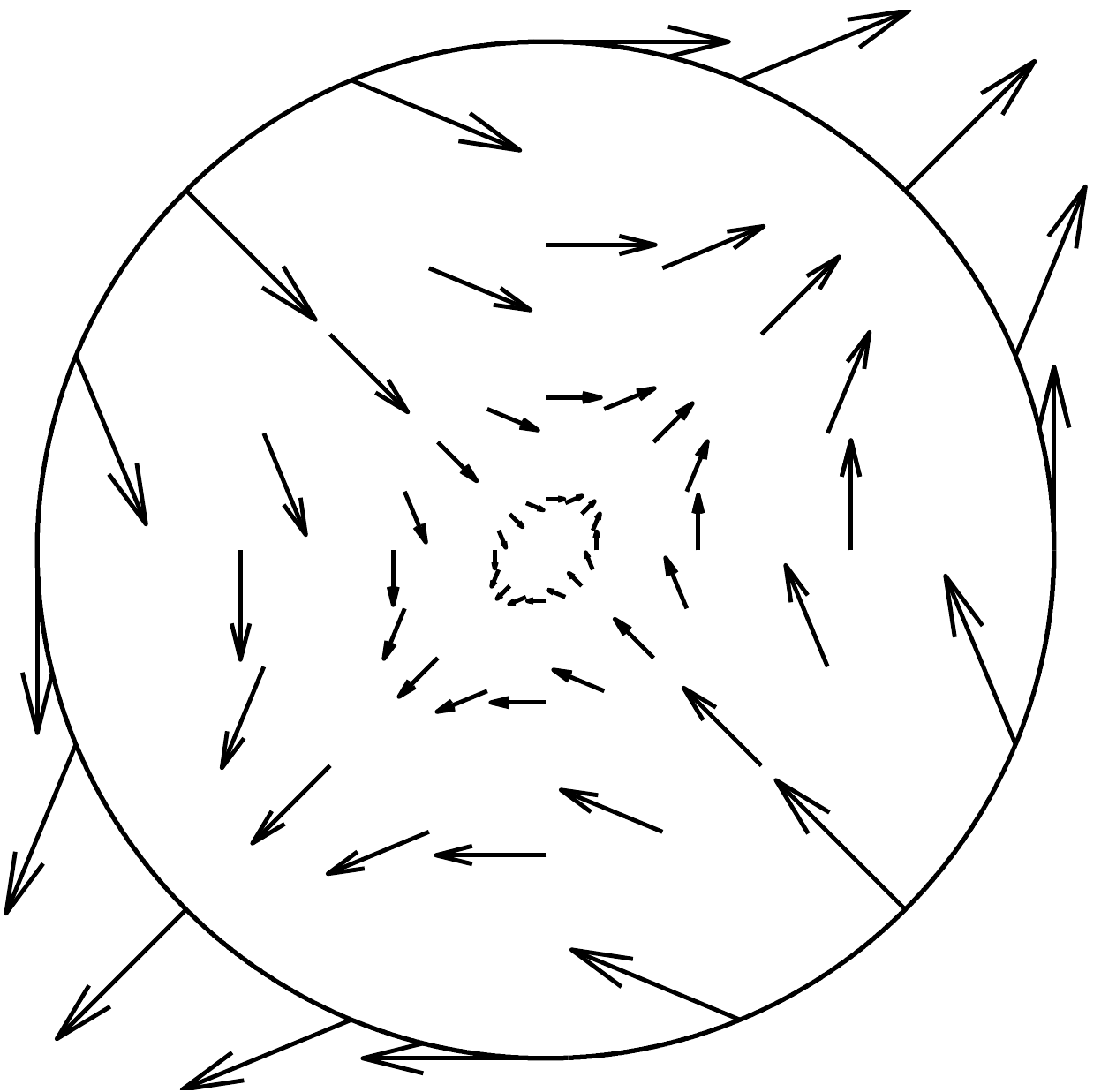}
	\mbox{c) $v_3(x,y)=(y,x)$} 
\end{minipage}
\caption{a) $\ind(v_1, N)=0$, $\ind(P_{\partial N}v_1, \, \partial_- N)=1$; b) $\ind(v_2, N)=1$, $\ind(P_{\partial N}v_2, \, \partial_- N)=0$; c) $\ind(v_3, N)=-1$, $\ind(P_{\partial N}v_3, \, \partial_- N)=2$.}
\label{fig:one}
\end{figure}

Identity~\eqref{eq:morse} can be seen as a generalization of the Poincar\'e-Hopf index formula.
As an immediate corollary, we obtain a necessary condition for the existence of nowhere vanishing vector fields which extends in $N$ a given a boundary datum.
 
 \begin{cor}
  Let $g\colon\partial N \to \R^d$ be a \emph{continuous} function, satisfying~\eqref{hp g VMO}, and assume that $P_{\partial N} g$ has finitely many zeros.
 If there exists a \emph{continuous} vector field $v$, satisfying~\eqref{v}, such that $v_{|\partial N} = g$ then
 \[
  \ind(P_{\partial N} g, \, \partial_- N) = \chi(N).
 \]
 \end{cor}
 This Corollary gives an answer to Question~\ref{quest: comb boundary} in case we consider smooth vector fields.
 
 \medskip
 Our aim in this paper is to extend Proposition~\ref{prop: morse index} to the $\VMO$ setting.
 For this purpose, we extend the definition of index to arbitrary $\VMO$ fields, with a trace at the boundary.
 We introduce another quantity, which we call ``inward boundary index'' and denote by $\ind_-(v, \, \partial N)$, playing the role of $\ind(P_{\partial N}v, \, \partial_- N)$. 
 (The reader is referred to Section~\ref{sect: VMO index} for the definitions). \par
 
 Then, our main result is
\begin{theor} \label{th: VMO morse}
 Let $N$ be a compact, connected and orientable submanifold of $\R^d$, with boundary.
 Let $g\in \VMO(\partial N, \, \R^d)$ be a boundary datum which fulfills
 \[
  g(x)\in T_x N \qquad \textrm{and} \qquad c_1 \leq \abs{g(x)} \leq c_2
 \]
 for some constants $c_1, \, c_2 > 0$ and $\H^{n - 1}$-a.e. $x\in \partial N$.
 If \mbox{$v\in \VMO(N, \, \R^d)$} is a map with trace~$g$ at the boundary, satisfying
 \[
  v(x) \in T_x N
 \]
 for a.e. $x\in N$, then
 \[
  \ind(v, \, N) + \ind_- (v, \, \partial N) = \chi(N) .
 \]
\end{theor}
Note that this Theorem is the analogous of Proposition~\ref{prop: morse index} for $\VMO$ vector fields. 
Finally, regarding Question ~\ref{quest: comb boundary}, we have the following answer. 
 
\begin{prop} \label{prop: CNS}
 Let $g\in \VMO(\partial N, \, \R^d)$ satisfy the assumption~\eqref{hp g VMO}.
 A field $v\in\VMO(N, \, \R^d)$ that satisfies~\eqref{v} and has trace $g$ exists if and only if
 \begin{equation}
 \label{eq:sufficient}
  \ind_-(g, \, \partial N) = \chi(N) . 
 \end{equation}
\end{prop}

In case the boundary datum satisfies~\eqref{hp g VMO},~\eqref{eq:sufficient} and~$g\in W^{1 - 1/p, p}(\partial N, \, \R^d)$ for some~\mbox{$p > 1$},
one can choose an extension~$v$ which, in addition to~\eqref{v}, satisfies~$v\in W^{1, p}(N, \, \R^d)$ (see Corollary~\ref{cor: sobolev}).
Therefore, the results we discuss in our paper are indeed relevant to the analysis of variational models for nematic shells.

We conclude this introduction with an outline of the paper. 
In Section~\ref{sect: VMO} we discuss some preliminary material, including basic properties of the $\VMO$ space (Subsection~\ref{subsect:VMO}), 
an application of Brezis and Nirenberg's degree theory to the existence of non-vanishing $\VMO$ fields on an unbounded manifold (Subsection~\ref{subsect: unbounded}),
and the definition of index in the basic case of a continuous field with a finite number of zeros (Subsection~\ref{subsect: cont index}).
In Appendix~\ref{app: cont index}, we recall how this notion can be extended to continuous fields with an arbitrary number of zeros, owing to Thom's Transversality Theorem. 
In Section~\ref{sect: VMO index}, by means of an approximation argument, this extension allows us to give a notion of index for a $\VMO$ vector field 
and to prove Theorem~\ref{th: VMO morse} and Proposition~\ref{prop: CNS}.
Finally, in Section~\ref{sec:application}, we apply these results to the existence of line fields with $\VMO$ regularity.
Interestingly, such an existence result shares the same topological obstruction as the existence result for vector fields.
As a side result of the existence of $\VMO$ $Q$-tensor fields, we obtain topological conditions for the existence of line fields with $\VMO$ regularity,
thus extending to this weaker setting a classical result due to Poincar\'e and Kneser.


\medskip
\noindent\textbf{Notation.} In the following sections either $N =\R^n$, or $N$ is a compact, connected and oriented manifold with boundary, of dimension~$n$, 
embedded as a submanifold of $\R^d$ for some $d\in\N$.

\begin{itemize}
 \item The injectivity radius of $N$ (see, e.g., do Carmo \cite{doCarmo}) is called $r_0$.
 \item We denote geodesic balls in $N$ by $B^N_r(x)$ or simply $B_r(x)$, when it is clear from the context that we work in $N$. In case $N = \R^n$, we write $B^n_r(x)$ or $B^n(x, r)$.
 \item For $\varepsilon > 0$, we set
 \[  
    N_\varepsilon := \left\{ x\in N \colon \dist(x, \, \partial N)\geq \varepsilon \right\}.
 \]
 \item For each $x\in\partial N$, we denote by $\nu(x)$ the outward unit normal to $\partial N$ in $T_xN$.
 \item Given a non-empty, convex and closed set $K\subset \R^d$, we denote the nearest-point projection on $K$ by $\proj_K$.
 \item Given a manifold $X\subset \R^d$ and a continuous map $v\colon X \to \R^d$, we denote the tangential component of $v$ by
 \[
  P_X v (x) := \mathrm{proj}_{T_x X} v(x) \qquad \textrm{for } x\in X .
 \]
\end{itemize}


 \section{Preliminary material}
 \label{sect: VMO}
 
 \subsection{VMO functions}
 \label{subsect:VMO}
 
 For the reader's convenience, we recall here the basic definitions about $\VMO$ functions, following the presentation of \cite{BN2} (to which the reader is referred, for more details). 
 All the functions we consider here take values in $\R^d$, so functional spaces such as, e.g., $L^1(N, \, \R^d)$ or $\VMO(N, \, \R^d)$ will be simply written as $L^1(N)$ or $\VMO(N)$.
 
 Recall that $N$ is endowed with a Riemannian measure~$\sigma$. For $u\in L^1(N)$ (with respect to~$\sigma$), define
 \begin{equation} \label{BMO norm}
  \norm{u}_{\BMO} := \underset{\substack{\varepsilon \leq r_0, \, x\in N_{2\varepsilon} \\ }}{\sup} 
  \fint_{B_\varepsilon(x)} \abs{u(y) - \bar u_\varepsilon(x) } \, \d \sigma(y) ,
 \end{equation}
 where
 \begin{equation} \label{medie}
  \bar u_\varepsilon (x) := \fint_{B_\varepsilon(x)} u(y) \, \d \sigma(y) , \qquad \textrm{for } x\in N_{2\varepsilon} .
 \end{equation}
 The set of functions with $\norm{u}_{\BMO} < +\infty$ will be denoted $\BMO(N)$, and~\eqref{BMO norm} defines a norm on $\BMO(N)$ modulo constants. 
 Using cubes instead of balls leads to an equivalent norm. 
 Moreover, if $\varphi\colon X_1 \to X_2$ is a $\mathscr C^1$ diffeomorphism between two unbounded manifolds, then $u\in \BMO(X_2)$ implies $u\circ \varphi\in \BMO(X_1)$ and
 \[
  \norm{u\circ\varphi}_{\BMO (X_1)} \leq C\norm{u}_{\BMO (X_2)}.
 \]
 Bounded functions (in particular, continuous functions) belong to BMO. Following Sarason, we define $\VMO(N)$ as the closure of $\mathscr C(N)$ with respect to the BMO norm. 
 Functions in $\VMO(N)$ can be characterized by means of this lemma (see \cite[Lemma 3]{BN1}):
 \begin{lemma} \label{lemma: VMO}
  A function $u\in\BMO(N)$ is in $\VMO(N)$ if and only if
  \[
   \lim_{\varepsilon\to 0} \sup_{x\in N_{2\varepsilon}} \fint_{B_\varepsilon(x)} \abs{u(y) - \bar u_\varepsilon(x)} \, \d \sigma(y) \to 0 .
  \]
 \end{lemma}
 
 Sobolev spaces provide an interesting class of functions in $\VMO$, since, for critical exponents, the embeddings which fail to be in $L^\infty$ hold true in $\VMO$:
 \[
  W^{s, p}(N) \subset \VMO(N) \qquad \textrm{whenever } 0 < s < n, \: sp = n.
 \]
 In general, $\VMO$ functions do not have a trace on the boundary.
 However, it is possible to introduce a subclass of $\VMO$ for which traces are well defined. We sketch here the construction.
 
 First, we need to embed $N$ as a domain of a bigger manifold $X$,  smooth and without boundary.
Here, we take $X$ as the double of $N$, that is, the manifold we obtain by gluing two copies of $N$ along their boundaries.
Modifying, if necessary, the value of $d$ we can assume that $X\subset \R^d$. 
Also, let $U$ be a tubular neighbourhood of $\partial N$ in $X$, and assume that the nearest-point projection $\pi\colon U \to \partial N$ is well defined. 
Now, we fix $g\in \VMO(\partial N)$ and we extend it to a function $G$, by the formula
 \begin{equation} \label{G}
  G(x) = \begin{cases}
          g(\pi(x)) \chi(x) & \textrm{if } x\in X \cap U \\
          0                 & \textrm{if } x\in X \setminus U
         \end{cases}
 \end{equation}
 where $\chi$ is a cut-off function, which is equal to $1$ near $\partial N$ and vanishes outside $U$. It can be checked that $G\in \VMO(X)$.
 
 We say that a function $u\in \VMO(N)$ has trace $g$ on $\partial N$, and we write $u\in\VMO_g(N)$, if and only if the function defined by
  \[
   \begin{cases}
    u & \textrm{in } N \\
    G & \textrm{in } X\setminus N
   \end{cases}
  \]
 is in $\VMO(X)$.
 This definition is independent on the choice of $\chi$ and of $X$ (see \cite[Property 6]{BN2}). The notion of $\VMO_g$ is stable under diffeomorphism: 
 suppose $\varphi\colon X_1 \to X_2$ is a $\mathscr C^1$ diffeomorphism between bounded manifolds, mapping diffeomorphically $\partial X_1$ onto $\partial X_2$ . 
 If $g\in \VMO(\partial X_2)$ and $u\in\VMO_g(X_2)$, then 
 \[
  u\circ\varphi \in \VMO_{g\circ\varphi}(X_1).
 \]
 As an example of $\VMO$ functions with trace, let us mention that every map in $W^{1, n}(X)$ has a trace in the sense of $\VMO$, which coincides with the Sobolev trace.


 \subsection{Combing an unbounded manifold in VMO}
 \label{subsect: unbounded}
 
 In this section, we prove Proposition~\ref{prop: unbounded}.
 Of course, it could be obtained as a corollary of our main result, Theorem~\ref{th: VMO morse}.
 Anyway, it can be proved independently, and we present here an elementary argument inspired by \cite[Theorem 2.28]{hatcher}.
 We assume that $N$ is a compact, connected $n$-manifold \emph{without boundary}, embedded as an hypersurface of $\R^{n + 1}$.
 
 \begin{proof}[Proof of Proposition~\ref{prop: unbounded}]
 It is well-known that, if $\chi(N) = 0$, then a nowhere vanishing, smooth (hence $\VMO$) vector field on $N$ exists.
 The idea of the proof is the following:  One picks an arbitrary continuous field, approximates it with a field $v$ having a finite number of zeros, 
 then uses the Poincar\'e-Hopf formula and the hypothesis $\chi(N) = 0$ to show that $\ind(v, \, N) = 0$, so $v$ can be modified into a nowhere vanishing field. 
  This argument is given in detail in the proof of Proposition~\ref{prop: CNS}, in case $N$ is a manifold with boundary, and it is even simpler when $\partial N = \emptyset$.
  
  Let us prove the other side of the proposition: we suppose that a tangent vector field \mbox{$v\in \VMO(N)$} such that $\infess_N\abs{v} > 0$ exists, and we claim that $\chi(N) = 0$.
  Every compact hypersurface of $\R^{n + 1}$ is orientable, so there is a smooth unit vector field $\gamma\colon N \to \R^{n + 1}$ such that $\gamma(x)\perp T_x N$ for all $x\in N$. 
  The choice of such a map induces an orientation on $N$, and $\gamma$ is called the Gauss map of the oriented manifold $N$. 
  We can also assume that $n$ is even, since $\chi(N) = 0$ whenever $N$ is a compact, unbounded manifold of odd dimension (see, e.g., \cite[Corollary 3.37]{hatcher}).
  
  Consider the function $H\colon N\times [0, \, \pi] \to \R^{n + 1}$ given by
  \[
   H(x, \, t) :=  (\cos t) \gamma(x) + (\sin t) w(x), 
  \]
  where $w:=\frac{v}{\vert v\vert}$. The function $H$ is clearly well defined since $\infess_N\abs{v} > 0$ by assumption. Moreover,
  it is readily checked that $\abs{H(x, t)}^2 = 1$ for all $(x, t)\in N\times [0, \, \pi]$. We claim that
  \begin{equation} \label{H cont}
   H\in \mathscr C \left([0, \, \pi], \, \VMO(N, \, \S^n) \right) .
  \end{equation}
  Indeed,  $H(\cdot, t)$ is the linear combination of functions in $\VMO(N)$ and hence belongs to $\VMO(N)$, for all $t$. On the other hand, for all $t_1, \, t_2\in [0, \pi]$
  \[
   \norm{H(\cdot, \, t_1) - H(\cdot, \, t_2)}_{\BMO} \leq \abs{\cos t_1 - \cos t_2} \norm{\gamma}_{\BMO} +  \abs{\sin t_1 - \sin t_2} \norm{w}_{\BMO} ,
  \]
  whence the claimed continuity~\eqref{H cont} follows.
  
  Since the degree is a continuous function $\VMO(N, \, \S^n) \to \mathbb Z$ (see \cite[Theorem 1]{BN1}), we infer that
  \[
   \deg(H(\cdot, \, 0), \, N, \, \S^n) = \deg(H(\cdot, \, \pi), \, N, \, \S^n) .
  \]
  On the other hand, $H(\cdot, \, 0) = \gamma$ and $H(\cdot, \, \pi) = -\gamma$. 
  By standard properties of the degree (in particular, \cite[Properties~(d, f) p. 134]{hatcher}), and since we have assumed that $n$ is even, we have
  \[
    \deg(-\gamma, \, N, \, \S^n) = (-1)^{n + 1} \deg(\gamma, \, N, \, \S^n) = - \deg (\gamma, \, N, \, \S^n) ,
  \]
  hence
  \[
   \deg(\gamma, \, N, \, \S^n) = - \deg(\gamma, \, N, \, \S^n) .
  \]
By the degree formula ~\eqref{deg_formula} and Gauss-Bonnet Theorem (see, e.g.,~\cite[page 196]{GuiPol74}), for an even-dimensional hypersurface $N$
 \[
 	\deg(\gamma, \, N, \, \S^n) = \deg(\gamma, \, N, \, \S^n) \fint_{\S^n} \d\sigma_n = \frac{1}{\omega_n}\int_N \gamma^*(\d\sigma_n)= \frac{1}{\omega_n}\int_N \kappa \,\d\sigma=    \frac12 \chi(N),
\]
where $\d\sigma_n$ is the volume form of $\S^n$, $\omega_n:=\int_{\S^n}\d\sigma_n$ is the volume of $\S^n$, and $\kappa$ is the Gaussian curvature of $N$.
Since $\deg(\gamma, \, N, \, \S^n)=0$ by the above construction, this shows that $\chi(N)=0$ and thus completes the proof.
 \end{proof} 
 
 \begin{remark}
  When $\chi(N)\neq 0$, Proposition~\ref{prop: unbounded} shows that there is no unit vector field in the critical Sobolev space $W^{s, p}(N)$, for $0 < s < n$ and $sp = n$.
  In contrast, when $sp < n$ it is not difficult to construct unit vector fields in $W^{s, p}(N)$.
  For instance, on $N = \S^{2k}$ one may consider a field with two ``hedgehog'' singularities, of the form $x \mapsto x/\abs{x}$, located at the opposite poles of the sphere.
 \end{remark}
 
 \subsection{The index and the inward boundary index}
 \label{subsect: cont index}
 
 Before studying the index in the VMO setting, we recall the definition of the index in the simplest setting, i.e. we work with smooth vector fields having a finite number of zeros.
 As we mentioned in the Introduction, we use local charts.
 Consider a smooth vector field on~$N$, i.e. a map~$v\colon N\to \R^d$ which satisfies~$v(x)\in T_x N$ for any~$x\in N$. We assume that
 \begin{equation} \label{v-no-zero-bd} 
  0 \notin v(\partial N) .
 \end{equation}
 Let $f \colon V \to \R^n$ be a chart defined on an open set $V\csubset N\setminus\partial N$.
 We consider the (smooth) map $f_* v\colon f(V) \subset \R^n \to \R^n$ defined by
 \begin{equation} \label{v-coordinates}
  f_*v(y) := \d f_{f^{-1}(y)}\left( v\circ f^{-1}(y) \right) \qquad \textrm{for } y\in f(V)\subset\R^n 
 \end{equation}
 (there is a slight abuse of notation here, because~$\d f_{f^{-1}(y)}$ is defined on~$T_{f^{-1}(y)} N$ whereas $v$ is an~$\R^d$-valued map). 
 For any~$x\in v^{-1}(0)$ and a chart~$f$ defined in a neighbourhood of~$x$, we assume that
 \begin{equation} \label{v-transverse} 
  \d (f_*v)_{f(x)} \textrm{ is invertible} .
 \end{equation} 
 If $f$, $g$ are two local charts around $x$, then $\d (f_*v)_{f(x)}$ is invertible if and only if $\d(g_*v)_{g(x)}$ is, condition~\eqref{v-transverse} is thus independent of the choice of the chart.
 Condition~\eqref{v-transverse} also implies that the set $v^{-1}(0)$ is discrete (by the local inversion theorem), hence is finite because $X$ is compact.
 Moreover, given two coordinate charts $f$ and $g$ which agree with the fixed orientation of $X$,
 the Jacobians $\det \d (f_*v)_{f(x)}$ and $\det \d (g_*v)_{g(x)}$ have the same sign. 
 Thus, if $U\subset X$ is an open set and $v$ is a transverse vector field on $X$ satisfying~\eqref{v-no-zero-bd}--\eqref{v-transverse}, the index of $v$ on $U$ is well-defined by the formula
 \begin{equation} \label{transv index}
  \ind(v, \, N) := \sum_{x\in v^{-1}(0)} \mathrm{sign } \det \d (f_* v)_{f(x)} .
 \end{equation}
 This formula can be expressed in an equivalent way.
 Pick a geodesic ball $B_r(x)\csubset U$ around each zero $x$, so small that no other zero is contained in $B_r(x)$.
 Then, $\frac{f_*v}{\abs{f_*v}}$ is well-defined as a map $\partial B_r(x)\simeq \S^{n - 1} \to \S^{n - 1}$, and 
 \begin{equation} \label{smooth index}
  \ind(v, \, N) = \sum_{x\in v^{-1}(0)} \deg \left(\frac{f_*v}{\abs{f_*v}}, \, \partial B_r(x), \, \S^{n - 1}\right) .
 \end{equation}
 The equivalence of~\eqref{transv index} and~\eqref{smooth index} follows, e.g., from \cite[Equation (4.1), p. 25]{BN2}.
 
 Given any continuous vector field~$v$ satisfying~\eqref{v-no-zero-bd}, one can approximate it with a sequence of smooth fields~$v_j$ satisfying~\eqref{v-no-zero-bd} and~\eqref{v-transverse}, 
 which converge to~$v$ in the uniform norm.
 Moreover, for large~$j$ the indices~$\ind(v_j, \, N)$ take a common value which depends only on~$v$.
 This allows to define the index for a continuous field satisfying only~\eqref{v-no-zero-bd} (see Definition~\ref{def: cont index}).
 This construction is a particular case of more general results, which are discussed e.g. in~\cite[Chapter~5]{hirsch}.
 For the convenience of the reader, we sketch it in Appendix~\ref{app: cont index}.
 
 To describe the behaviour of~$v$ at the boundary, we need to introduce another quantity. 
 For any $x\in \partial N$, denote with $\nu(x)$ the outward unit normal to $\partial N$ in $T_x N$. Then, we introduce the set
 \begin{equation} \label{inward bd}
  \partial_-N[v] := \left\{x\in\partial N \colon v(x)\cdot\nu(x) < 0 \right\} ,
 \end{equation}
 called the inward boundary, which is open in $\partial N$. (We simply write  $\partial_- N$, when $v$ is clear from the context).
 The tangential component $P_{\partial N} v$ defines a vector field over $\partial_- N$ and, despite $0 \notin v(\partial N)$, it is possible that $P_{\partial N} v$ vanishes at some point.
 However, $P_{\partial N} v$ does not vanish on~$\partial(\partial_- N)$. Indeed, 
 \[
  \partial (\partial_- N) = \left\{x\in\partial N\colon v(x)\cdot \nu(x) = 0 \right\} ,
 \]
 hence if $x\in\partial (\partial_-N)$ we have $P_{\partial N} v (x) = v(x) \neq 0$.
 Thus, for any continuous field~$v$ satisfying~\eqref{v-no-zero-bd} we can define the inward boundary index by
  \[
  \ind_- (v, \, \partial N) := \ind(P_{\partial N}v, \, \partial_- N),
  \]
 where the right-hand side has a meaning in view of Definition~\ref{def: cont index}.
 Notice that the inward boundary index depends only on $v_{|\partial N}$.
 Hence, it make sense to compute it for a continuous map~$g$ defined only on $\partial N$, provided that $g$ is tangent to $N$ and vanishes nowhere.
 
 The index and the inward boundary index satisfy some properties, such as excision, homotopy invariance, and stability with respect to uniform convergence. 
 For future reference, we provide a quantitative statement for the latter property. The following proposition is proved in Appendix~\ref{app: cont index}.
 
 \begin{prop} \label{prop: stability}
  Let $v$ be a continuous field on~$N$ satisfying~\eqref{v-no-zero-bd}.
  There exists $\varepsilon_1=\varepsilon_1(v)>0$ such that, for any other continuous vector field~$w$ satisfying~\eqref{v-no-zero-bd}, if
  \begin{equation*}
   \norm{v - w}_{\mathscr C(\partial N)} < \varepsilon_1
  \end{equation*}
  then~$\ind(v, \, N) = \ind(w, \, N)$ and~$\ind_-(v, \, \partial N) = \ind_-(w, \, \partial N)$. For example, an admissible choice of $\varepsilon_1$ is 
  \[
    \varepsilon_1:=\frac{\sqrt 5 -1}{4}\min_{\partial N}{|v|}.
  \]
 \end{prop}
 
 Morse's index formula (see Proposition~\ref{prop: morse index}) holds true for arbitrary continuous fields. 

 \begin{prop} \label{prop: continuous index formula}
  Let $v$ be a continuous vector field on $N$, such that $0\notin v(\partial N)$. Then,
  \[
   \ind(v, \, N) + \ind_-(v, \, \partial N) = \chi(N) .
  \]
 \end{prop}
 
 This is a special case of a more general Poincar\'e-Hopf type formula in~\cite{Pugh}. For completeness, we present its proof  in Appendix~\ref{app: cont index}.



 \section{The index in the VMO setting}
 \label{sect: VMO index}

 \subsection{Proof of Theorem~\ref{th: VMO morse}}
 \label{subsect: main-theorem}
 
In this section we define the index of a $\VMO$ field and then we prove our main results.
From now on, $X$ will be taken to be the topological double of $N$, as in Section~\ref{sect: VMO}.
Moreover, throughout this section we consider a function $g\in \VMO(\partial N)$ such that
\begin{equation} \label{hp g}
  g(x)\in T_x N  \qquad \textrm{and} \qquad  c_1 \leq \abs{g(x)} \leq c_2 \qquad \textrm{for } \H^{n - 1}\textrm{-a.e. } x\in \partial N
\end{equation}
for some constants $c_1, \,c_2 > 0$.
Let $v$ be a $\VMO$ vector field with trace $g$, that is,
\begin{equation} \label{hp v}
 v\in \VMO_g(N), \qquad v(x)\in T_xN \qquad \textrm{for } \textrm{a.e. } x\in N.
\end{equation}
By definition of $\VMO_g(N)$, the function $u$ given by
\[
 u:= \begin{cases}
      v & \textrm{on } N \\
      G & \textrm{on } X\setminus N ,
     \end{cases}
\]
where $G$ is the extension of $g$ defined in~\eqref{G}, is in $\VMO(X)$.
 Denote the local averages of $u$ and $g$ by
\[ 
  \bar u_\varepsilon (x) := \fint_{B_\varepsilon^X(x)} u(y) \, \d \sigma(y) , \qquad \textrm{for } x\in X.
 \]
 and
\[
  \bar g_\varepsilon (x) := \fint_{B_\varepsilon^{\partial N}(x)} g(y) \, \d \H^{n-1}(y) , \qquad \textrm{for } x\in \partial N.
\]
Consider the functions 
\begin{equation}
\label{def:uege}
 u_\varepsilon := P_X \bar u_\varepsilon \qquad \textrm{and} \qquad g_\varepsilon := P_X \bar g_\varepsilon, 
\end{equation}
defined on $X$ and $\partial N$, respectively, which are continuous and tangent to $X$.
As we will prove in the following Lemma~\ref{lemma: g_epsilon positive}, Lemma~\ref{lemma: u_eps bd}, 
and Lemma~\ref{lemma: constant}, the quantities $\ind(u_\varepsilon, \, N)$ and $\ind_-(g_\varepsilon, \, \partial N)$ are well-defined and constant with respect to $\varepsilon$, for $\varepsilon$ small enough.

\begin{defn} \label{def: VMO index}
 Given $g\in \VMO(\partial N)$ and $v$ which satisfy~\eqref{hp g}--\eqref{hp v}, we define the index and the inward boundary index of $v$ by
 \[
  \ind(v, \, N) := \ind(u_\varepsilon, \, N) \qquad \textrm{and} \qquad \ind_-(v, \, \partial N) := \ind_-(g_\varepsilon, \, \partial N) ,
 \]
 where $\varepsilon$ is fixed arbitrarily in $(0, \, \varepsilon_0)$ and $\varepsilon_0$ is given by Lemma~\ref{lemma: constant}.
\end{defn}

Once we have checked that the index, in the sense of Definition~\ref{def: VMO index}, is well-defined, Theorem~\ref{th: VMO morse} will follow straightforwardly from 
Proposition~\ref{prop: continuous index formula}.
However, before directing our attention to the main theorem, there are some facts which need to be checked. 
 
The next two lemmas compare the behaviour of $g_\varepsilon$ and ${u_\varepsilon}_{|\partial N}$.

\begin{lemma} \label{lemma: g_epsilon positive}
 For every $\delta > 0$, there exists $\varepsilon_0 \in (0, \, r_0)$ so that, for all $\varepsilon\in (0, \, \varepsilon_0)$ and all $x\in \partial N$, we have
 \[
  c_1 - \delta \leq \abs{g_\varepsilon (x)} \leq c_2 + \delta .
 \]
\end{lemma}

\begin{lemma} \label{lemma: u_eps bd}
 It holds that
 \[
  \lim_{\varepsilon\to 0} \sup_{x\in \partial N} \abs{u_\varepsilon(x) - g_\varepsilon(x)} = 0 .
 \]
\end{lemma}

Combining Lemmas~\ref{lemma: g_epsilon positive} and~\ref{lemma: u_eps bd}, we deduce that there exist constants $\varepsilon_0$,$c > 0$ such that
\[
 \abs{u_\varepsilon(x)} \geq c, \quad \abs{g_\varepsilon(x)} \geq c \qquad \textrm{for all } \varepsilon\in (0, \, \varepsilon_0) \textrm{ and all } x\in \partial N.
\]
In particular,
\[
0 \notin u_\varepsilon(\partial N) \qquad \textrm{and} \qquad 0 \notin g_\varepsilon(\partial N)
\]
so $\ind(u_\varepsilon, \, N)$ and $\ind_-(g_\varepsilon, \, \partial N)$ are well-defined 
for all $\varepsilon \in (0, \, \varepsilon_0)$.
\medskip

Before proving Lemmas~\ref{lemma: g_epsilon positive} and~\ref{lemma: u_eps bd}, we need a useful property.

\begin{lemma} \label{lemma: overline}
 It holds that
 \[
  \lim_{\varepsilon\to 0}\sup_{x\in X} \abs{u_\varepsilon(x) - \bar u_\varepsilon(x)} = \lim_{\varepsilon\to 0}\sup_{x\in \partial N} \abs{g_\varepsilon(x) - \bar g_\varepsilon(x)} = 0 .
 \]
\end{lemma}
\proof
We present the proof for $u_\varepsilon$ only, as the same argument applies to $g_\varepsilon$ as well.
Consider a \emph{finite} atlas $\mathscr A = \left\{U_\alpha\right\}_{\alpha\in A}$ for $X$ and, for each $\alpha\in A$, 
let $\nu_1^\alpha, \, \ldots, \, \nu^\alpha_{d - n}$ be a smooth moving frame for the normal bundle of $X$, defined on $U_\alpha$ 
(i.e., $\left(\nu^\alpha_i(y)\right)_{1 \leq i \leq d - n}$ is an orthonormal base for $T_y X^\perp$, for all $y\in U_\alpha$).
Set
\begin{equation} \label{curvatura}
 C_N := \max_{\substack{\alpha\in A \\ 1 \leq i \leq d - n}} \norm{D \nu^\alpha_i}_{L^\infty(U_\alpha)} < +\infty .
\end{equation}
For all $\alpha\in A$ and $x\in U_\alpha$, we write
\begin{equation} \label{overline1}
 u_\varepsilon(x) - \bar u_\varepsilon(x) = \sum_{i = 1}^{d - n} \left(\bar u_\varepsilon(x)  \cdot \nu^\alpha_i(x) \right) \nu^\alpha_i(x) 
\end{equation}
and, since $u(y)\cdot \nu_i^\alpha(y) = 0$ for a.e. $y\in U_\alpha$, we have
\[
 \bar u_\varepsilon(x) \cdot \nu^\alpha_i(x) = \fint_{B_\varepsilon(x)} u(y) \cdot \left(\nu^\alpha_i(x) - \nu^\alpha_i(y) \right) \, \d \sigma(y) .
\]
Taking into account~\eqref{curvatura}, we infer
\[
 \abs{\bar u_\varepsilon(x) \cdot \nu^\alpha_i(x)} \leq C_N \fint_{B_\varepsilon(x)} \abs{u(y)} \abs{x - y} \d\sigma(y) .
\]
To bound the right-side of this inequality, we exploit the injection $\BMO(X) \hookrightarrow L^p(X)$, which holds true for all $1 \leq p < +\infty$, and the H\"older inequality. For a fixed $p$, we obtain
\begin{equation} \label{overline2}
 	\abs{\bar u_\varepsilon(x) \cdot \nu_i^\alpha(x)} 
		\leq C_N \sigma(B_\varepsilon(x))^{-1} \norm{x - y}_{L^p(B_\varepsilon(x))} \norm{u}_{L^{p^\prime}(X)} 
		\leq C_{N, n, p} \, \varepsilon^{1 + n/p - n} \norm{u}_{L^{p^\prime}(X)} ,
\end{equation}
for some constant $C_{N, n, p}$ depending only on $C_N$, $n$ and $p$.
Whenever $p^\prime< +\infty$, the $L^{p^\prime}$ norm of~$u$ can be bounded using only the BMO norm of~$u$ and $\fint_X u$ (with the help of \cite[Lemmas~A.1 and B.3]{BN1}).
Thus, choosing $p = p(n) > 1$ so small that $1 + n/p - n > 0$, from~\eqref{overline1} and~\eqref{overline2} we conclude the proof.
\endproof
\begin{proof}[Proof of Lemma~\ref{lemma: g_epsilon positive}]
Setting
\[
 S_x := \left\{ v\in T_xN \colon c_1 \leq \abs{v} \leq c_2 \right\}
\]
we have, for all $x\in \partial N$,
\begin{equation} \label{dist g}
 \dist(\bar g_\varepsilon(x), \, S_x) \leq \fint_{B_\varepsilon(x)} \abs{\bar g_\varepsilon(x) - g(y)} \, \d\sigma(y) + \fint_{B_\varepsilon(x)} \dist(g(y), \, S_x) \, \d\sigma(y) .
\end{equation}
The first term in the right-hand side tends to zero as $\varepsilon\to 0$, uniformly in $x$, due to Lemma~\ref{lemma: VMO}. On the other hand, it holds
\begin{equation} \label{geom conv}
 \sup_{\substack{x, \, y\in \partial N \\ \dist(x, \, y) \leq \varepsilon}} \: \sup_{v\in S_y} \dist(v, \, S_x) \longrightarrow 0 \qquad \textrm{as } \varepsilon \to 0,
\end{equation}
since $N$ is compact and smooth up to the boundary. Formula~\eqref{geom conv} can be easily proved, e.g., by contradiction: Assume that~\eqref{geom conv} does not hold. 
Then, we find a number $\eta > 0$, a sequence $(\varepsilon_k)_{k\in\N}$ of positive numbers s.t. $\varepsilon_k\searrow 0$, two sequences $(x_k)_{k\in\N}$, $(y_k)_{k\in\N}$ in $N$ 
and one $(v_k)_{k\in\N}$ in $\R^d$, which satisfy
\[
 v_k\in S_{y_k}, \qquad \dist(x_k, \, y_k) \leq \varepsilon_k , \qquad \dist(v_k, \, S_{x_k}) \geq \eta.
\]
By compactness of $N$, up to subsequences we can assume that
\[
 x_k \to x\in N, \qquad y_k \to y\in N, \, \qquad v_k \to v\in \R^d ,
\]
where $c_1 \leq \abs{v} \leq c_2$. Let $\nu_i, \, \nu_2, \, \ldots , \, \nu_{d - n}$ be a moving frame for the normal bundle of $N$, defined on a neighbourhood of $y$.
Passing to the limit in the condition
\[
 v_k \cdot \nu_i(y_k) = 0 \qquad \textrm{for all } i
\]
we find that $v\in T_yN$, hence $v\in S_y$. But $y = x$, because $\dist(x_k, \, y_k) \leq \varepsilon_k \to 0$. Thus, we have found $v\in S_x$ so that $\dist(v, \, S_{x_k})\geq \eta/2 > 0$.
On the other hand, if $\varphi\colon U\subset N \to \R^n$ is a coordinate chart near $x$ then
\[
  w_k := \d \varphi_{\varphi(x_k)}^{-1} \left(\d \varphi_x v\right), 
  \qquad \tilde w_k := \min\{\max\{\abs{w_k}, \, c_1\}, \, c_2\} \frac{w_k}{\abs{w_k}}
\]
are well-defined for $k\gg 1$ and $\tilde w_k\in S_{x_k}$, $\tilde w_k\to v$. This leads to a contradiction.

 Thus, we can take advantage of~\eqref{hp g} and~\eqref{geom conv} to estimate the second term in the right-hand side of~\eqref{dist g}.
 We deduce that
 \[
  \sup_{x\in \partial N} \dist(\bar g_\varepsilon(x), \, S_x) \longrightarrow 0 \qquad \textrm{as } \varepsilon \to 0
 \]
 and, invoking Lemma~\ref{lemma: overline}, we conclude the proof.
\end{proof}

\begin{proof}[Proof of Lemma~\ref{lemma: u_eps bd}]
In view of Lemma~\ref{lemma: overline}, proving that
\[
 \lim_{\varepsilon\to 0} \sup_{x\in \partial N} \abs{\bar u_\varepsilon(x) - \bar g_\varepsilon(x)} = 0
\]
is enough to conclude. In addition, it holds
\begin{equation} \label{u_eps bd 1}
 \abs{\bar u_\varepsilon(x) - \bar g_\varepsilon(x)} \leq \abs{\bar u_\varepsilon(x) - \bar G_\varepsilon(x)}  + \abs{\bar G_\varepsilon(x) - \bar g_\varepsilon(x)} ,
\end{equation}
so we can study each term in the right-hand side and prove that they converge to zero as $\varepsilon\to 0$.

Let us focus on the first term. We remark that $\bar u_\varepsilon - \bar G_\varepsilon = \overline{(u - G)}_\varepsilon$ and that
\[u - G = \begin{cases}
           v - G & \textrm{on } N \\
           0     & \textrm{on } X \setminus N .
          \end{cases}
\]
Thus, for all $x\in \partial N$ we have
(recall that $(u-G)(y) = 0$ for almost any $y\in X\setminus N$.)
\begin{align*}
 \frac{\sigma\left(B^X_\varepsilon(x) \setminus N\right)}{\sigma\left(B^X_\varepsilon(x)\right)}  \abs{\overline{(u - G)}_\varepsilon (x)} & 
 \le \frac{1}{\sigma\left(B_\varepsilon^X(x)\right)} \int_{B^X_\varepsilon(x) \setminus N} \abs{(u - G)(y) - \overline{(u - G)}_\varepsilon(x)} \, \d\sigma(y) \\
 &\leq \fint_{B^X_\varepsilon(x)} \abs{(u - G)(y) - \overline{(u - G)}_\varepsilon(x)} \, \d\sigma(y) ,
\end{align*}
where $\sigma$ is the Riemannian measure on $X$.
 Now, assume for a while that 
 there exist two numbers~$\alpha, \, \varepsilon_0 > 0$ such that
\begin{equation} \label{densita bordo}
 \frac{\sigma\left(B^X_\varepsilon(x) \setminus N\right)}{\sigma\left(B^X_\varepsilon(x)\right)} \geq \alpha
\end{equation}
for all $x\in\partial N$ and all $\varepsilon\in (0, \, \varepsilon_0)$.
 Therefore, when $\varepsilon < \varepsilon_0$ we deduce
\[
 \sup_{x\in\partial N} \abs{\overline{(u - G)}_\varepsilon (x)} \leq \alpha^{-1} \sup_{x\in\partial N} \fint_{B^X_\varepsilon(x)} \abs{(u - G)(y) - \overline{(u - G)}_\varepsilon(x)} \, \d\sigma(y)
\]
and, since $u - G\in \VMO(X)$, the right-hand side tends to $0$ as $\varepsilon\to 0$, by Lemma~\ref{lemma: VMO}. To conclude, we have to prove the validity of~\eqref{densita bordo}. To this end,
we assume without loss of generality that $N$ is a smooth, bounded domain in $X = \R^n$.
For a fixed $x_0\in \partial N$, we can locally write $\partial N$ as the graph of a smooth function $\varphi\colon B_{r_0}(0) \subseteq \R^{n - 1}\to \R$.
Then, letting $L_{x_0}(x) := \varphi(x_0) + \d \varphi(x_0) (x - x_0)$ be the linear approximation of $\varphi$, considering the region between the graphs of $\varphi$ and $L_{x_0}$ we deduce
\[
 \abs{\H^n\left(N \cap B^n_\varepsilon(x_0)\right) - \frac12 \H^n\left(B^n_\varepsilon(x_0)\right)} \leq \int_{B_\varepsilon^{n - 1}(x_0)} \abs{\varphi(x) - L_{x_0}(x)} \, \d x .
\]
By the Taylor-Lagrange formula, we have $\abs{\varphi(x) - L_{x_0}(x)} \leq M\abs{x - x_0}^2$, for a suitable constant $M$ controlling the hessian of $\varphi$. Thus
\[
 \abs{\H^n\left(N \cap B^n_\varepsilon(x_0)\right) - \frac12 \H^n\left(B^n_\varepsilon(x_0)\right)} \leq M\omega_n \varepsilon^{n + 2} ,
\]
where $\omega_n := \H^{n - 1}(\S^{n - 1}) = n\H^n(B^n_1(0))$, and
\begin{equation} \label{taylor palle}
 \abs{\frac{\H^n\left(N \cap B^n_\varepsilon(x_0)\right)}{\H^n\left(B^n_\varepsilon(x_0)\right)} - \frac12 } \leq n M \varepsilon^2 .
\end{equation}
The constant $M$ depends on $\varphi$, which is defined just locally, in a neighbourhood of $x_0$.
Nevertheless, owning to the compactness of $\partial N$, one needs to consider a \emph{finite} number of functions $\varphi$ only, and hence it is possible to choose a constant $M$ 
which satisfies~\eqref{taylor palle} for all $x_0\in N$. Therefore,~\eqref{densita bordo} follows. 
\medskip

Now, we have to deal with the second term in~\eqref{u_eps bd 1}.
We can assume without loss of generality that $X = \R^n$ and
\[
 N = \R^n_+ := \left\{(x_1, \, x_2, \, \ldots, \, x_n)\in \R^n \colon x_1 \geq 0 \right\} .
\]
We can always reduce to this case by composing with local coordinates, with the help of a partition of the unity argument.
For the sake of simplicity, denote the variable in $\R^n$ by $x = (t, \, y)$, where $t\in\R$ and $y\in\R^{n - 1}$.

Call $\alpha_n$ the volume of the unit ball of $\R^n$. 
Using Fubini's theorem and the definition~\eqref{G} of~$G$, for $x_0 = (0, \, y_0)$ and $\varepsilon$ small enough (so that $\chi(t,y)\equiv 1$, for $|t|\leq\varepsilon$) we compute 
\begin{align*}
 	\bar G_\varepsilon(x_0) 
		&= \frac{1}{\alpha_n\varepsilon^n} \int_{-\varepsilon}^{\varepsilon}  
	 		\left(\int_{B^{n - 1}(y_0, \, \sqrt{\varepsilon^2 - t^2})}  G(t, \, y)\, \d y\right) \d t\\ 
 		&= \frac{\alpha_{n - 1}}{\alpha_n\varepsilon^n} \int_{-\varepsilon}^{\varepsilon}
			\left(\varepsilon^2 - t^2\right)^{\frac{n - 1}{2}}\left(\fint_{B^{n - 1}(y_0, \, \sqrt{\varepsilon^2 - t^2})} g(y) \, \d y\right) \d t\\ \\
		&= \frac{\alpha_{n - 1}}{\alpha_n\varepsilon^n} \int_{-\varepsilon}^\varepsilon 
			\left(\varepsilon^2 - t^2\right)^{\frac{n - 1}{2}} \bar g_{\sqrt{\varepsilon^2 - t^2}} (y_0) \, \d t \\
		&=  \frac{\alpha_{n - 1}}{\alpha_n\varepsilon^n} \int_{-1}^1 
			\left(\varepsilon^2 - (\varepsilon s)^2\right)^{\frac{n - 1}{2}} \bar g_{\sqrt{\varepsilon^2 - (\varepsilon s)^2}} (y_0) \varepsilon\, \d s \\
 		&= \frac{\alpha_{n - 1}}{\alpha_n} \int_{-1}^1 
			\left(1 - s^2\right)^{\frac{n - 1}{2}} \bar g_{\varepsilon\sqrt{1 -s^2}} (y_0) \, \d s .
\end{align*}
On the other hand, Fubini's theorem also implies that
\[
 \alpha_n = \alpha_{n - 1} \int_{-1}^1 \left(1 - t^2\right)^{\frac{n - 1}{2}} \, \d t ,
\]
thus
\begin{equation} \label{u_eps bd 2}
 \abs{\bar G_\varepsilon(x_0) - \bar g_\varepsilon(x_0)} \leq 
 \frac{\alpha_{n - 1}}{\alpha_n} \int_{-1}^1 \left(1 - t^2\right)^{\frac{n - 1}{2}} \abs{\bar g_{\varepsilon\sqrt{1 - t^2}} (y_0) - \bar g_\varepsilon(y_0)} \, \d t .
\end{equation}
For all $-1 < t < 1$, since $B^{n - 1}(y_0, \, \varepsilon\sqrt{1 - t^2}) \subset B^{n - 1}(y_0, \, \sqrt{1 - t^2})$ we infer that
\begin{align*}
 \abs{\bar g_{\varepsilon\sqrt{1 - t^2}} (y_0) - \bar g_\varepsilon(y_0)} &\leq \fint_{B^{n - 1}(y_0, \, \varepsilon\sqrt{1 - t^2})} \abs{g(y) - \bar g_\varepsilon(y_0)} \, \d y \\
 &\leq \left(1 - t^2\right)^{\frac{1 - n}{2}} \fint_{B^{n - 1}(y_0, \, \varepsilon)} \abs{g(y) - \bar g_\varepsilon(y_0)} \, \d y 
\end{align*}
and, injecting this information into~\eqref{u_eps bd 2}, we deduce
\[
 \abs{\bar G_\varepsilon(x_0) - \bar g_\varepsilon(x_0)} \leq \frac{2\alpha_{n - 1}}{\alpha_n} \fint_{B^{n - 1}(y_0, \, \varepsilon)} \abs{g(y) - \bar g_\varepsilon(y_0)} \, \d y .
\]
Hence, applying once again Lemma~\ref{lemma: VMO}, we conclude that the second term in the right-hand side of~\eqref{u_eps bd 1} converges to zero as $\varepsilon\to 0$, uniformly in $x\in\partial N$.
\end{proof}

\begin{remark}
 Setting
\[
 \ind_-(g, \, \partial N) := \ind_- (u_\varepsilon, \, \partial N)
\]
gives another possibility to define the inward boundary index of $g$, just as natural as our Definition~\ref{def: VMO index}.
However, thanks to Lemma~\ref{lemma: u_eps bd} and to the stability of the inward boundary index 
(Proposition~\ref{prop: stability}), we deduce that the two definitions agree.
\end{remark}

\begin{lemma} \label{lemma: constant}
 There exists $\varepsilon_0 \in (0, \, r_0)$ so that the functions
 \[
  \varepsilon \mapsto \ind(u_\varepsilon, \, N), \qquad \varepsilon \mapsto \ind_-(g_\varepsilon, \, \partial N) 
 \]
 are constant on $(0, \, \varepsilon_0)$.
\end{lemma}
\proof
We have already remarked that $\ind(u_\varepsilon, \, N)$ and $\ind_-(g_\varepsilon, \, \partial N)$ are well-defined for $\varepsilon$ small, as a consequence of Lemmas~\ref{lemma: g_epsilon positive} 
and~\ref{lemma: u_eps bd}.
Consider the functions $H \colon N \times (0, \varepsilon_0) \to \R^d$ and $G \colon \partial N \times (0, \, \varepsilon_0) \to \R^d$ given by
\[
 H(x, \, \varepsilon) := u_\varepsilon(x) = \proj_{T_x N} \bar u_\varepsilon(x)
\]
and
\[
  G (x, \, \varepsilon) := g_\varepsilon(x) = \proj_{T_x N} \bar g_\varepsilon(x) .
\]
These maps are well-defined and continuous.
Indeed, it follows from the dominated convergence theorem that $(x, \, \varepsilon)\mapsto \bar u_\varepsilon(x)$ and $(x, \, \varepsilon)\mapsto \bar g_\varepsilon(x)$ are continuous, 
whereas the family of projections $\mathrm{proj}_{T_x N}$ depends continuously on $x$.
Applying 
Proposition~\ref{prop: stability} to~$H$ and~$G$, we conclude that $\ind(u_\varepsilon, \, N)$ and $\ind_-(g_\varepsilon, \, \partial N)$ are constant with respect to $\varepsilon$. 
\endproof

After these preliminary lemmas, the proof of our main result, Theorem~\ref{th: VMO morse}, is straightforward.

\begin{proof}[Proof of Theorem~\ref{th: VMO morse}]

Let $v$ be given and let $u_\varepsilon$ and $g_\varepsilon$ be the continuous approximants of $v$ and of its trace, as defined in~\eqref{def:uege}. 
Let $\varepsilon_0>0$ be the constant given by Lemma~\ref{lemma: constant}. Up to choosing a smaller value of $\varepsilon_0$, owing to Lemma~\ref{lemma: g_epsilon positive} and Lemma~\ref{lemma: u_eps bd}, 
for all $\varepsilon \in (0,\varepsilon_0)$ and all  $x\in \partial N$ we have
\begin{align*}
 	\abs{u_\varepsilon(x)} \geq \frac{c_1}{2}, \quad \abs{g_\varepsilon(x)} \geq \frac{c_1}{2},\\
	|u_\varepsilon(x)-g_\varepsilon(x)|< \frac{\sqrt 5 -1}{8}c_1.
\end{align*} 
Therefore, by Proposition~\ref{prop: stability}, for all $\varepsilon \in (0,\varepsilon_0)$ there holds
\begin{equation}
\label{eq:quasi}
	\ind_-(u_\varepsilon, \, \partial N) = \ind_-(g_\varepsilon, \, \partial N).
\end{equation}
To conclude, by Definition~\ref{def: VMO index} and Proposition~\ref{prop: continuous index formula} we obtain
\begin{align*}
  \ind(v, \, N) + \ind_-(v, \, \partial N) &= \ind(u_\varepsilon, \, N) + \ind_-(g_\varepsilon, \, \partial N) \\
		&\stackrel{\eqref{eq:quasi}}{=} \ind(u_\varepsilon, \, N) + \ind_-(u_\varepsilon, \, \partial N) = \chi(N),
\end{align*}
which proves Theorem~\ref{th: VMO morse}.
\end{proof}

\subsection{Proof of Proposition~\ref{prop: CNS}}

This subsection aims at proving~Proposition~\ref{prop: CNS}.
In particular, given a boundary datum~$g\in\VMO(\partial N)$ which satisfies~\eqref{hp g} and the topological condition~\eqref{eq:sufficient}, we extend it to a non vanishing VMO field defined on~$N$.

In the following lemma, we work out the construction near the boundary.
For any~$r > 0$ small enough, the set
\[
 U_r := \left\{ x\in N \colon \dist(x, \, \partial N) < r \right\}
\]
is a tubular neighbourhood of~$\partial N$.
In particular, there exists an orientation-preserving diffeomorphism~$\varphi\colon\partial N \times [0, \, r]\to \overline U_r$ 
such that~$\dist(\varphi(y, \, s), \, \partial N) = s$ for any~$(y, \, s)\in \partial N\times [0, \, r]$.
Then,~\mbox{$C_r := \varphi(\partial N\times\{r\})$} is a submanifold of~$N$, diffeomorphic to~$\partial N$.
We define the function~$\bar v\colon \overline U_r \to \R^d$ by
\begin{equation} \label{interpolating v}
 (\bar v\circ\varphi)(y, \, s) := \bar{g}_s(y) = \fint_{B^{\partial N}_s(y)} g(z) \, \d \H^{n - 1} (z)
\end{equation}
for any~$(y, \, s)\in \partial N\times [0, \, r]$, and set~$v := P_X \bar v$.  Thus $v$ is a tangent vector field. Moreover, it 
satisfies

\begin{lemma} \label{lemma: interpolation}
 There exists~$r>0$ such that the following properties hold.
 \begin{enumerate}[label = (\roman*)]
  \item \label{item: tubular} The set~$U_r$ is a tubular neighbourhood of~$\partial N$.
  \item \label{item: VMO} We have~$v\in \VMO(U_r)$, and~$v$ has trace~$g$ on~$\partial N$ (in the sense of Brezis and Nirenberg, as defined in Section~\ref{sect: VMO}).
  \item \label{item: inw} The function~$v$ is continuous on~$\overline U_r\setminus \partial N$, $v(x)\neq 0$ for every~$x\in \overline U_r$ and
  \[
   \ind_-(g, \, \partial N) = \ind_-(v, \, C_r) .
  \]
 \end{enumerate}
\end{lemma}
\begin{proof}
By Lemmas~\ref{lemma: g_epsilon positive} and~\ref{lemma: constant}, we can pick~$r$ such that~\ref{item: tubular} holds and, in addition,
\begin{equation} \label{interpolation1}
 \frac{c_1}{2} \leq |g_s| \leq  2c_2 , \qquad \ind_-(g_s, \, \partial N) = \ind_-(g, \, \partial N)
\end{equation}
for any~$0 < s \leq r$.
The field~$\bar v$ is continuous on~$\overline U_r\setminus\partial N$, due to the dominated convergence theorem, so~$v$ is continuous on~$\overline U_r\setminus\partial N$.
Taking a smaller~$r$ if necessary, from~\eqref{interpolation1} and Lemma~\ref{lemma: overline} we deduce that
\begin{equation} \label{interpolation2}
 \frac{c_1}{3} \leq \abs{v} \leq 3c_2 \qquad \textrm{in } \overline U_r .
\end{equation}
Moreover, there holds
\[
 \begin{split}
  \norm{(v\circ\varphi)(\cdot, \, r) - g_r}_{\mathscr C(\partial N)} &\leq 
  \norm{\proj_{T_{\varphi(\cdot, r)}X} - \proj_{T_{\varphi(\cdot, 0)}X} }_{\mathscr C(\partial N)} \norm{\bar g_r }_{\mathscr C(\partial N)} \\
  &\stackrel{\eqref{hp g}}{\leq} C \norm{\proj_{T_{\varphi(\cdot, r)}X} - \proj_{T_{\varphi(\cdot, 0)}X} }_{\mathscr C(\partial N)} .
 \end{split}
\]
Since~$X$ is a smooth, compact manifold, $\proj_{T_{\varphi(\cdot, r)}X}$ converges uniformly to~$\proj_{T_{\varphi(\cdot, 0)}X}$ as~$r\to 0$, so
\[
 \norm{(v\circ\varphi)(\cdot, \, r) - g_r}_{\mathscr C(\partial N)} \to 0 \qquad \textrm{as } r\to 0 .
\]
By the stability of the boundary index (Proposition~\ref{prop: stability}), we obtain that
\[
 \ind_-((v\circ\varphi)(\cdot, \, r), \, \partial N) = \ind_-(g_{r}, \, \partial N)
\]
if~$r$ is small enough.
On the other hand, the index and the boundary index are invariant by composition with a diffeomorphism.
(For smooth vector fields satisfying~\eqref{v-no-zero-bd} and~\eqref{v-transverse}, this follows by Formula~\eqref{transv index}; for arbitrary continuous fields satisfying~\eqref{v-no-zero-bd}, one argues by density.)
Therefore, we conclude that
\[
 \ind_-(v, \, C_r) = \ind_-(g_{r}, \, \partial N) \stackrel{\eqref{interpolation1}}{=} \ind_-(g, \, \partial N) ,
\]
and~\ref{item: inw} holds true.

We only need to check that~$v\in\VMO(U_r)$, with trace~$g$ on~$\partial N$; this is equivalent to proving that the map
\begin{equation*} 
  u := \begin{cases}
    v & \textrm{on } U_r \\
    G & \textrm{on } X \setminus N 
 \end{cases}
\end{equation*}
belongs to~$\VMO(U_r\cup (X\setminus N))$. (Here~$G\in\VMO(X)$ denotes the standard extension of~$g$, as defined in~\eqref{G}.)
By~\cite[Theorem~1, Eq.~(1.2)]{BN2}, this is also equivalent to 
\begin{equation} \label{interpolation4}
 \sup_{x\in W_{r - 2\varepsilon}} I_\varepsilon(u, \, x) \to 0 \qquad \textrm{as } \varepsilon\to 0 ,
\end{equation}
where we have set
\[
 I_\varepsilon(u, \, x) := \fint_{B^X_\varepsilon(x)} \fint_{B^X_\varepsilon(x)} \abs{u(y) - u(z)} \, \d\sigma(y) \, \d\sigma(z)
\]
and~$W_s := U_s\cup (X\setminus N)$, for any~$s$.
Thanks to~\cite[Lemma~7]{BN1}, we know that~$\bar v \in\VMO(U_r)$ has trace~$g$ at the boundary, that is the map
\[
 \bar u:= \begin{cases}
                   \bar v & \textrm{on } U_r \\
                   G & \textrm{on } X \setminus N
               \end{cases}
\]
belongs to~$\VMO(W_r)$. This yields 
\begin{equation} \label{interpolation5}
 \sup_{x\in W_{r - 2\varepsilon}} I_\varepsilon(\bar u, \, x) \to 0 \qquad \textrm{as } \varepsilon\to 0
\end{equation}
(again by~\cite[Theorem~1]{BN2}). For a fixed~$0 < s < r$ and~$0 < \varepsilon < (r - s)/2$, there holds
\[
 \sup_{x\in W_{r - 2\varepsilon}} I_\varepsilon(u, \, x) \leq \max\left\{ \sup_{x\in W_{s - \varepsilon}} I_\varepsilon(\bar u, \, x) + 2 \norm{\bar u - u}_{L^\infty(U_s)}, \, 
 \sup_{x\in U_{r - 2\varepsilon}\setminus U_{s - \varepsilon}} I_\varepsilon(u, \, x) \right\} .
\]
We take the upper limit as~$\varepsilon\to 0$. Using~\eqref{interpolation5} and the fact that~$u$ is continuous on~$N\setminus \partial N$, we deduce
\[
 \limsup_{\varepsilon\to 0}\sup_{x\in W_{r - 2\varepsilon}} I_\varepsilon(u, \, x) \leq 2 \norm{\bar u - u}_{L^\infty(U_s)} .
\]
Now, we let~$s\to 0$. By applying Lemma~\ref{lemma: overline}, we conclude that~\eqref{interpolation4} holds, so~$u\in\VMO(W_r)$.
\end{proof}

 We can finally give the proof of Proposition~\ref{prop: CNS}.
 \begin{proof}[Proof of Proposition~\ref{prop: CNS}] 
 	Since a field $v\in \VMO_{g}(N)$ satisfying~\eqref{v} has $\ind(v,\, N)=0$, Theorem~\ref{th: VMO morse} directly implies~\eqref{eq:sufficient}.
 	In order to prove the converse implication, let a field $g \in \VMO(\partial N)$ be given, such that~\eqref{hp g VMO} and~\eqref{eq:sufficient} hold.
Let~$v\colon U_r\to\R^d$ be the field defined by~\eqref{interpolating v}, where~$r >0$ is given by Lemma~\ref{lemma: interpolation}.
Let~$N_r:= N\setminus U_r$, and let~$V\colon N_r\to\R^d$ be any continuous field such that $V = v$ on~$C_r = \partial N_r$.
(For instance, one can take as~$V$ the standard extension of~$v_{|C_r}$, as defined by~\eqref{G}.)
As $0 \notin V(C_r)$, by the Transversality Theorem~\ref{th: transversality} there exists a smooth tangent vector field~$F$ on~$X$ such that $F$ has finitely many zeros in $N_r$,
$F_{|C_r}=v_{|C_r}$ and, by stability (Corollary~\ref{cor: stability}) and by Theorem~\ref{th: VMO morse}, that
\[
\begin{split}
    \ind(F, \, N_r) &= \ind(V, \, N_r) = \chi(N) - \ind_-(v, \, C_r) \\
    &\!\!\stackrel{\ref{item: inw}}{=} \chi(N) - \ind_-(g, \, \partial N) \stackrel{\eqref{eq:sufficient}}{=} 0.
\end{split}
\]
Let~$\tilde F$ be defined by~$\tilde F = v$ on~$U_r$ and~$\tilde F = F$ on~$N \setminus U_r$.
The field~$\tilde F$ is continuous in the interior of~$N$, belongs to~$\VMO_g(N)$ and satisfies~$F(x) \neq 0$ for any~$x\in U_r$, by Lemma~\ref{lemma: interpolation}.
Assume for the moment that $F(x) \neq 0$ for all $x\in N_r$, and set 
 \[
 	A_1 :=\{x \in N \colon |\tilde F(x)| < c_1\},   \qquad 	A_2 :=\{x \in N \colon |\tilde F(x)|>c_2\}.
 \] 
Then, the field defined by
\[
	V(x):= \left\{
		\begin{array}{ll}
			\displaystyle \tilde F(x)\frac{c_i}{|\tilde F(x)|} & \text{if }x\in A_i,\ i=1,2,	\\
			\tilde F(x)		& \text{otherwise}.
		\end{array} 
		\right.
\]
belongs to $\VMO_{g}(N)$ and satisfies~\eqref{v}.

To conclude, we note that there is a standard technique to modify a continuous field $u$ such that~$0 \notin u(\partial N)$,
\[
	\ind(u, \, N)=0,\qquad \text{and}\qquad  \# \{x \in N \colon u(x)=0\} < +\infty
\] 
into a continuous field $\tilde u$ such that $|\tilde u|>0$ and $\tilde u =u$ on $\partial N$. 
(We will apply this technique to~$u = F$ over~$N\setminus U_r$.)
We sketch here the idea. First, up to a continuous transformation, we can assume that all the zeros are contained in one coordinate neighbourhood $U$, 
with chart $\phi\colon U\subseteq N \to D \subseteq \R^n$, so we can reduce to study the vector field in coordinates:
let $D=B_1(0),\ D_{1/2}=B_{1/2}(0)$,  assume that $u\colon D\to \R^{n}$ and $|u|>0$ in $D \setminus D_{1/2}$. Then, 
\[
	0= \ind(u, \, D)=\deg \left(\frac{u}{\abs{u}}, \, \partial D, \, \S^{n - 1}\right)
\]
and there exists a continuous field $\psi\colon D \to \S^{n-1}$ such that $\psi_{|\partial D}=\frac{u}{|u|}.$
\[
	\tilde \psi(x):=\left\{
		\begin{array}{ll}
			\psi(x) & \text{if }x\in D_{1/2}\\
			\psi(x)(2\dist(x,\partial D)+(1-2\dist(x,\partial D)|u(x)|)) & \text{if }x\in  D \setminus D_{1/2},
		\end{array}
		\right.	
\]
so that $\tilde \psi(x)$ is continuous on $D$, nowhere zero, and it agrees with $u$ on $\partial D$. To conclude, the field 
\[
	\tilde u(x):=\left\{
		\begin{array}{ll}
			u & \text{if }x\in N\setminus U\\
			\phi^* \tilde \psi & \text{if }x\in  U
		\end{array}
		\right.	
\]
is continuous and nowhere zero on $N$.  Here $\phi^* \tilde \psi(x):=\d \phi^{-1}_{\phi(x)}\tilde \psi (\phi(x))$ denotes the usual pullback of $\tilde \psi$ via~$\phi$.
\end{proof}
 
 \begin{cor} \label{cor: sobolev}
  Let~$p> 1$ and $c_1,c_2>0$ and let~$g\in W^{1 - 1/p, p}(\partial N)$ be a vector field satisfying~ 
  \[
   g(x)\in T_x N, \quad  \quad c_1 \leq \abs{g(x)} \leq c_2, \,\,\,\hbox{ and }\,\,\,  \ind_-(g, \, \partial N) = \chi(N).
  \]
  Then, there exists~$v\in W^{1, p}(N)$ which satisfies~
  \[
  v(x)\in T_x N, \quad  \quad c_1 \leq \abs{v(x)} \leq c_2, 
  \]
  and has trace~$g$ at the boundary.
 \end{cor}
 \begin{proof}
 The case~$p = +\infty$ follows by classical facts on Lipschitz functions, therefore we assume that~$1 < p < +\infty$. 
 The extension~$v\in\VMO_{g}(N)$ we have constructed in the previous proof is actually continuous in the interior of~$N$ and smooth in~$N_\varepsilon$.
 Therefore, the corollary will be proved if we show that the field~$v$, defined by~\eqref{interpolating v}, belongs to~$W^{1, p}(U_\varepsilon)$ when~$g\in W^{1 - 1/p, p}(\partial N)$.
 By using a partition of unity and composing with local diffeomorphisms, we can assume with no loss of generality that~$U_\varepsilon = \Sigma_\varepsilon \times [0, \, \varepsilon]$, 
 where~$\Sigma_\varepsilon := [\varepsilon, \, 1 - \varepsilon]^{n - 1}\subset\R^{n - 1}$ is endowed with the norm~$\|x\| := \max_i |x_i|$.
 Then, Formula~\eqref{interpolating v} reduces to
 \begin{equation} \label{extension v}
  v(x) = \frac{1}{(2x_n)^{n - 1}} \int_{x_1 - x_n}^{x_1 + x_n} \ldots \int_{x_{n - 1} - x_n}^{x_{n -1} + x_n}  g(\xi_1, \, \ldots, \, \xi_{n - 1}) \, \d\xi_{n - 1} \ldots \, \d \xi_1 .
 \end{equation}
 Gagliardo, in his paper~\cite{Gagliardo}, used functions of this form\footnote{Actually, Gagliardo considered a function of the form
 \[
  v_*(x) = \frac{1}{x_n^{n - 1}} \int_{x_1}^{x_1 + x_n} \ldots \int_{x_{n - 1}}^{x_{n -1} + x_n} \, g(\xi_1, \, \ldots, \, \xi_{n - 1}) \, \d\xi_{n - 1} \ldots \d \xi_1,
 \]
 but his computations can be adapted to~$v$ defined by~\eqref{extension v} in a straightforward way.} to prove the existence of a right inverse for the trace operator~$W^{1, p}(\Omega)\to W^{1 - 1/p, p}(\partial\Omega)$.
 More precisely, he proved that
 \[
 \norm{v}_{W^{1, p}(U_\varepsilon)} \leq C \norm{g}_{W^{1 - 1/p, p}(\Sigma_0)} .
 \]
 This shows that~$v\in W^{1, p}(U_\varepsilon)$ as soon as~$g\in W^{1 - 1/p, p}(\Sigma_0)$, and concludes the proof.
 \end{proof}
 
 
\begin{remark} \label{remark:shaftingen}
 In our main results, Theorem~\ref{th: VMO morse} and Proposition~\ref{prop: CNS}, the boundary datum~$g$ and the field~$v$ are assumed to satisfy inequalities such as
 \[
  c_1 \leq \abs{g(x)} \leq c_2 \qquad \textrm{for a.e. } x \textrm{ and positive constants } c_1, c_2
 \]
 (see~\eqref{hp g},~\eqref{hp v}). The lower bound is the natural generalization of the condition~$g(x) \neq 0$, which makes no sense as the field~$g$ is not defined pointwise everywhere.
 On the other hand, the upper bound is a technical assumption, which is used only in the proof of Lemma~\ref{lemma: g_epsilon positive} to control the last term in~\eqref{dist g}.
 For our purposes, this assumption is not restrictive, since we are interested mainly in unit vector fields.
 However, after our results were announced, Jean Van Schaftingen remarked that the upper bound on~$g$ is unnecessary.
 For if $x, \, y\in N$ are close enough to each other and $v\in T_y N$, then one can prove
 \begin{equation*}
  \dist(v, \, T_x N) \leq C\abs{v} \dist(x, \, y) .
 \end{equation*}
 With the help of this fact and of Jensen inequality, we obtain
 \[
  \begin{split}
  \fint_{B^{\partial N}_\varepsilon(x)} \dist(g(y), \, T_x N) \, \d\sigma(y) &\leq C \varepsilon \fint_{B^{\partial N}_\varepsilon(x)} \abs{g(y)} \, \d \sigma(y) 
  \leq C \varepsilon \left(\fint_{B^{\partial N}_\varepsilon(x)} \abs{g(y)}^{n} \, \d \sigma(y)\right)^{1/n} \\
  &\leq C \varepsilon^{1/n} \norm{g}_{\VMO(\partial N)} \to 0 ,
  \end{split}
 \]
 where the last inequality follows by the continuous embedding $\VMO(\partial N)\hookrightarrow L^{n}(\partial N)$.
 Then, Lemma~\ref{lemma: g_epsilon positive} follows.
\end{remark}

 \section{An application: $Q$-tensor fields and line fields.}
 \label{sec:application}
 
  \newcommand{\Sz}{\mathscr S_0}
 \newcommand{\Q}{\mathbf Q}
 \renewcommand{\P}{\mathbb P}
 
In the mathematical modelling of Liquid Crystals two different theories are eminent. 
In the Frank-Oseen theory the molecules are represented by the unit vector field $n$ which appears in the energy~\eqref{eq:energy}. 
The main drawback of this approach is to neglect the natural head-to-tail symmetry of the crystals.  
The theory of Landau-de Gennes takes this symmetry into account by introducing a tensor-valued field, called \emph{Q-tensor}, 
to which is associated a scalar parameter~$s$ that represents the local average ordered/disordered state of the molecules. 
In the particular, but physically relevant, case when the order parameter is a positive constant, there is a bijection betweenQ-tensors and line fields. 
The differences between the vector-based and the line field-based theory have been studied in \cite{ball_zarnescu}, in two- and three-dimensional Euclidean domains. 
In this Section we have two aims: firstly we apply some of the results obtained in Section~\ref{sect: VMO index} to line fields on a compact 
surface, obtaining the VMO-analogue of Poincar\'e-Kneiser Theorem (see Theorem~\ref{theor:cont_line} below). 
In particular, we restrict to the case of a surface without boundary. Clearly, it would be 
interesting to extend Theorem~\ref{th: VMO morse} to VMO line fields. Moreover, we show how the question of orienting a line field, 
studied in \cite{ball_zarnescu}, has generally a negative answer on a compact surface. 
As it happens for liquid crystals in Euclidean domains, the elastic part 
 of the Landau-de Gennes energy for nematic shells is, at least in some simplified situations, 
 proportional to a Dirichlet type energy. See on this regard \cite{KRV11} and \cite{NapVer12E}.
Therefore, owing to the embedding of Sobolev spaces in VMO spaces,~\eqref{eq:embedding}, Proposition~\ref{prop:ultima} establishes a relation between the existence of 
finite energy Q-tensors with strictly positive order parameter and the topology of the underlying surface, thus extending our application scope from the Frank-Oseen theory 
to the (constrained) Landau-de Gennes one, for uniaxial nematic shells.

 
  \subsection{Q-tensors and line fields} 
 Nematic shells are the datum of a compact, connected
 and without boundary surface $N\subset \R^3$ coated with a thin film
  of rod-shaped, head-to-tail symmetric particles of nematic liquid crystal.
 At a given point $x\in N$, the local configuration is represented by
  a probability measure $\mu_x$ on the unit circle $C_x$ in $T_x N$.
 More precisely, for each Borel set $A \subset C_x$, $\mu_x(A)$ is
  the probability of finding a particle at $x$, with direction contained in $A$.
 To account for the symmetry of the particles, we require
  \begin{equation} 
  \label{eq:symmtr}
  \mu_x(A) = \mu_x(-A)
 \end{equation}
 for each Borel set $A \subset C_x$.
 Due to this constraint, the first-order momentum of $\mu_x$ vanishes.
 Hence, we are naturally led to consider the second-order momentum
 \begin{equation} \label{def Q}
  Q = \sqrt 2 \int_{C_x} \left(p^{\otimes 2} -  \frac12 \P_x \right) \, \d\mu_x(p) ,
 \end{equation}
 where $(p^{\otimes 2})_{ij}:=p^i p^j$ and $\P_x$ denotes the orthogonal projection on $T_xN$. Note that 
 $Q$ has been suitably renormalized, so that $Q = 0$ when $\mu_x$ is the uniform measure, and $|Q|=1$ when $\mu_x$ is a Dirac measure concentrated on one direction (see~\eqref{eq:normQ} and~\eqref{eq:Qxi}).
 This formula defines a real $3\times 3$ symmetric and traceless matrix called \emph{$Q$-tensor}. 
 As we are interested in fields on surfaces, we replaced the usual three-dimensional renormalization term $-\frac13 Id$ by  $-\frac12 \P_x$ (see, e.g., \cite{KRV11}). 
Once we have fixed an orientation on $N$, we let $\gamma$ denote the Gauss map. By definition~\eqref{def Q}, $Q\gamma(x) = 0$,  which translates
the intuitive fact that the probability of finding a particle in the normal direction of the surface is zero. 
We call this type of anchoring a {\itshape degenerate (tangent) anchoring} (see \cite{NapVer12E}).

For any $x\in N$ we define the class of ``admissible tensors'' at $x$ as
 \begin{equation} \label{Q_x}
  \Q_x := \left\{ Q\in \Sz \colon Q\gamma(x) = 0 \right\},
 \end{equation}
 where $\Sz$ is the space of $3\times 3$ real, symmetric, and traceless matrices, endowed with the scalar product $Q \cdot P = \sum_{ij} Q_{ij}P_{ij}$.
 It is clear from the definition that $\Q_x$ is a linear subspace of $\Sz$ of dimension $2$ (this can be easily checked, e.g., by proving that the map $\Sz \to \R^3$ given by 
 $Q \mapsto Q \gamma(x)$ is surjective).
 Moreover, $\Q_x$ varies smoothly with $x$.
 
 \begin{lemma} \label{lemma: Q bundle}
  The set
  \[
   \Q := \coprod_{x\in N} \Q_x ,
  \]
  equipped with the natural projection $(x, \, Q) \mapsto x$, is a smooth vector bundle on $N$.
 \end{lemma}
 \proof
 Consider a smooth orthonormal frame $(n, \, m, \, \gamma)$ defined on a coordinate neighbourhood of $N$, where $(n, \, m)$ is a basis for the tangent bundle of $N$.
 With straightforward computations, one can see that the matrices
 \[
  X_{ij} := n_i n_j - m_i m_j, \, \qquad Y_{ij} := n_i m_j + m_i n_j,
 \]
 \[
  E_{ij} := \gamma_i\gamma_j - \frac13 \delta_{ij}, \qquad F_{ij} := n_i \gamma_j + \gamma_i n_j, \qquad G_{ij} := m_i\gamma_j + \gamma_i m_j
 \]
 define an orthogonal frame for $\Sz$. Moreover, $(X(x), \, Y(x))$ is a basis for $\Q_x$, at each point $x$
 (see \cite{INSZ} for a use of this basis with a particular choice for $(n,m,\gamma)$).
 The lemma follows easily.
 \endproof

We can now analyze the special structure of the matrices in $\Q_x$. 
 Fix $Q\in \Q_x$,  from~\eqref{Q_x} it follows that $\gamma(x)$ is an eigenvector of $Q$, corresponding to the zero eigenvalue.
 Since $Q$ is symmetric and traceless, there exists an orthonormal basis $(n, \, m)$ of $T_x N$, whose elements are eigenvectors of $Q$, and the corresponding eigenvalues are opposite.
 Thus, denoting by $n$ the eigenvector corresponding to the positive eigenvalue, $Q$ can be written in the form 
 \begin{equation}
 \label{eq:useless}
  Q = \frac s2 \left(n^{\otimes 2} - m^{\otimes 2} \right)
 \end{equation}
 for some $s \geq 0$ (If $s=0$, then $Q=0$ and any choice of $n$ is allowed).
 Using the identity $n^{\otimes 2} + m^{\otimes 2} = \P_x$, we conclude that for each $Q\in \Q_x$ there exist a number $s\geq 0$ and a unit vector $n\in T_x N$ such that
 \begin{equation} 
 \label{representation}
  Q = s \left(n^{\otimes 2} - \frac12 \P_x \right) .
 \end{equation}
 The number $s$, called the order parameter, is uniquely determined, and from~\eqref{eq:useless} we obtain
 \begin{equation}
 \label{eq:normQ}
  \abs{Q}^2 =Q\cdot Q = \frac{s^2}{4}\left(n^{\otimes 2} - m^{\otimes 2} \right)\cdot \left(n^{\otimes 2} - m^{\otimes 2} \right) 
  =\frac{s^2}{4}\left(n^{\otimes 2}\cdot n^{\otimes 2}+ m^{\otimes 2}\cdot m^{\otimes 2} \right)=\frac{s^2}{2}.
 \end{equation}
 When $Q \neq 0$, $n$ is also uniquely determined, up to a sign.
 Thus, each $Q\in \Q_x \setminus \{0\}$ identifies a positive number and a (un oriented) direction in $T_x N$, that is, a \emph{line field}.
 
 A \textit{line field} on $N$ (also called \textit{1-distribution}) is an assignment of a (non zero) tangent direction --- but not an orientation --- to each point
  of the submanifold $N$.
 More precisely, following \cite[Chapter 6]{Spivak} a line field $L$ is a function
  that assigns to each point $x$ of a manifold $N$ a one-dimensional subspace $L(x) \subset T_xN$.  Then $L$ is spanned by a vector field \textit{locally}; 
  that is, we can choose a vector field $v$ such that $0\neq v(x)\in L(x)$ for all $x$ in some neighbourhood of $x$. 
  We say that $L$ is a smooth (continuous) 1-distribution if the vector field $v$ can be chosen to be smooth (continuous) in a neighbourhood of each point. 

 Conversely, to a given line $\ell\subset T_x N$ generated by a unit vector $\xi\in T_x N$
 it is possible to associate the measure 
 $\mu_x := \frac 12 \delta_\xi + \frac 12 \delta_{-\xi}$ and thus by~\eqref{def Q}
%
the direction $\xi$ corresponds to
 \begin{equation}
 \label{eq:Qxi}
  	Q = \sqrt2 \left(\xi^{\otimes 2} -  \frac12 \P \right) ,
 \end{equation}
 which is a unit $Q$-tensor.
 The reason for associating to the direction $\xi$ the measure $\mu_x 
 = \frac 12 \delta_\xi + \frac 12 \delta_{-\xi}$, instead of simply $\delta_\xi$, is to be found in the head-to-tail symmetry of the molecules expressed by~\eqref{eq:symmtr}.
%
 Thus, line fields on $N$ can be identified with sections of the bundle $\Q$, having modulus one.
 
 In the following, we relax the condition $\abs{Q}=1$, by requiring $|Q|$ to be bounded and uniformly positive.

 \subsection{Existence of $\VMO$ line fields} 
 
In what follows, we assume that $N\subset \R^3$ is a smooth, compact, connected 
 surface, without boundary.
 Based on Proposition~\ref{prop: unbounded} and on the results of Section~\ref{sect: VMO index}, 
 in Proposition~\ref{prop:ultima} 
we prove that the existence of
 a $\VMO$ line field is subject to the same topological 
 obstruction that holds for continuous vector fields.
 If we restrict to the continuous setting, 
the following result is classical (see, e.g., \cite[Theorem 2.4.6, p. 24]{HectorHirsch})
 \begin{theor}[Poincar\'e-Kneiser]
 \label{theor:cont_line}
 Let $N$ be a compact, connected submanifold of $\R^{n+1}$. Then
 a continuous line field exists if and only if $\chi(N) = 0$.
 \end{theor}

 \begin{defn} \label{def: Q}
  A $\VMO$ line field on $N$ is a map $Q\in \VMO(N, \, \Sz)$, such that
  \begin{equation} \label{Q}
  Q(x) \in \Q_x \qquad \textrm{and} \qquad c_1 \leq \abs{Q(x)} \leq c_2
 \end{equation}
 for some constants $c_1, c_2 > 0$ and $\H^2$-a.e. $x\in N$. 
 \end{defn}
 
 The condition $Q\in \VMO(N, \, \Sz)$ makes perfectly sense, 
 because $\Sz\simeq \R^5$ is a finite-dimensional linear space.

 \begin{prop}
 \label{prop:ultima}
 If a $\VMO$ line field on $N$ exists, then $\chi(N) = 0$, that is, $N$ has genus $1$.
 \end{prop}
 
 \proof
 The proof is based on the arguments of Section~\ref{sect: VMO index}, with straightforward adaptations.
 We approximate $Q$ with a family of continuos functions, by setting
 \[
  \bar Q_\varepsilon (x) := \fint_{B^n_\varepsilon(x)} Q(y) \, \d\sigma(y)
 \]
 for each $x\in N$ and $\varepsilon\in (0, \,r_0)$.
 Then, we define
 \[
  Q_\varepsilon(x) := \proj_{\Q_x} \bar Q_\varepsilon(x) \qquad \textrm{for } x\in N.
 \]
 The functions $Q_\varepsilon$ are continuous, since the $Q_x$'s vary smoothly (see Lemma~\ref{lemma: Q bundle}).
 Owing to~\eqref{Q}, and arguing as in Lemma~\ref{lemma: g_epsilon positive}, 
  it can be proved that
 \[
  \frac{c_1}{2} \leq \abs{Q_\varepsilon(x)} \leq 2c_2
 \]
 for all $x\in N$ and $\varepsilon$ small enough.
 In view of formula~\eqref{representation}, each $Q_\varepsilon$ induces a continuous line field on $N$.
 In fact, the continuity of $Q_\varepsilon$ gives the continuity of $\vert Q_\varepsilon\vert$. Consequently, we have that
 $s$ is a continuous function, thanks to~\eqref{eq:normQ}. On the other hand,
 the representation formula
~\eqref{representation} gives that 
 $$
  n^{\otimes 2}(x) = \frac{Q(x)}{s(x)} + \frac12 \P_x,
 $$
 which implies the continuity of $n^{\otimes 2}$ thanks to the assumed strict positivity 
 of $s$ and thanks to the continuity of the projection operator. The tensor $n^{\otimes 2}$
 is the line field we were looking for. 
 Thus, by Theorem~\ref{theor:cont_line}, it must be $\chi(N) = 0$.
 \endproof
 \bigskip

 \subsection{Orientability of line fields} 
 A typical problem in the study of line fields is to understand in which circumstances
 a Q-tensor can be described in terms of a vector, that is 
when, given a tensor field $Q$ with a specified regularity, one can find
 a unit vector field $n$ with the same regularity, such that (in three dimensions)
 \begin{equation}
 \label{eq:orientation}
 	Q = s \left(n^{\otimes 2} - \frac13 \mathbb{I} \right) 
 \end{equation}
 for some positive constant $s$. In other words, we are trying to prescribe an orientation
 for the Q-tensor without
 creating artificial discontinuities in the vector $n$. 
 If for a given tensor $Q$ we can find a vector $n$ for which the 
 representation~\eqref{eq:orientation} holds, we say that $Q$ is 
 {\itshape orientable}, otherwise {\itshape non-orientable}.
 The problem of the orientability of a Q-tensor 
 has been addressed and solved by Ball and Zarnescu in
  \cite{ball_zarnescu}, in the case of two- and three-dimensional Euclidean domains.
They showed that 
 the conditions for 
 orienting a given tensor field are of topological as well as of
 analytical nature. Precisely, they require a Sobolev-type regularity,
 i.e. $Q\in W^{1,p}(\Omega)$ with $pge 2$, together with
 the condition that the domain $\Omega$ be
 simply connected.  
 
Regarding Q-tensor fields on manifolds (which we assume here to be compact, connected, without boundary), we observe that
 there exists no two-dimensional surface $N$ and exponent $p\geq 2$ such that 
\[
	Q\in W^{1,p}(N)\quad \Rightarrow \text{Q is orientable}.
\]
Indeed, by Proposition~\ref{prop: unbounded} 
 the only surface which allows for the existence
 of a unit vector field with regularity at least $W^{1,2}$ is the torus, which is not simply connected, and on which  
simple examples of smooth nonorientable line fields can easily be constructed (see Fig.~\ref{fig_one}).
\begin{figure}[ht] 
\centering 
\parbox{.46\textwidth}{
	\centering
	\includegraphics[height=3.5cm]{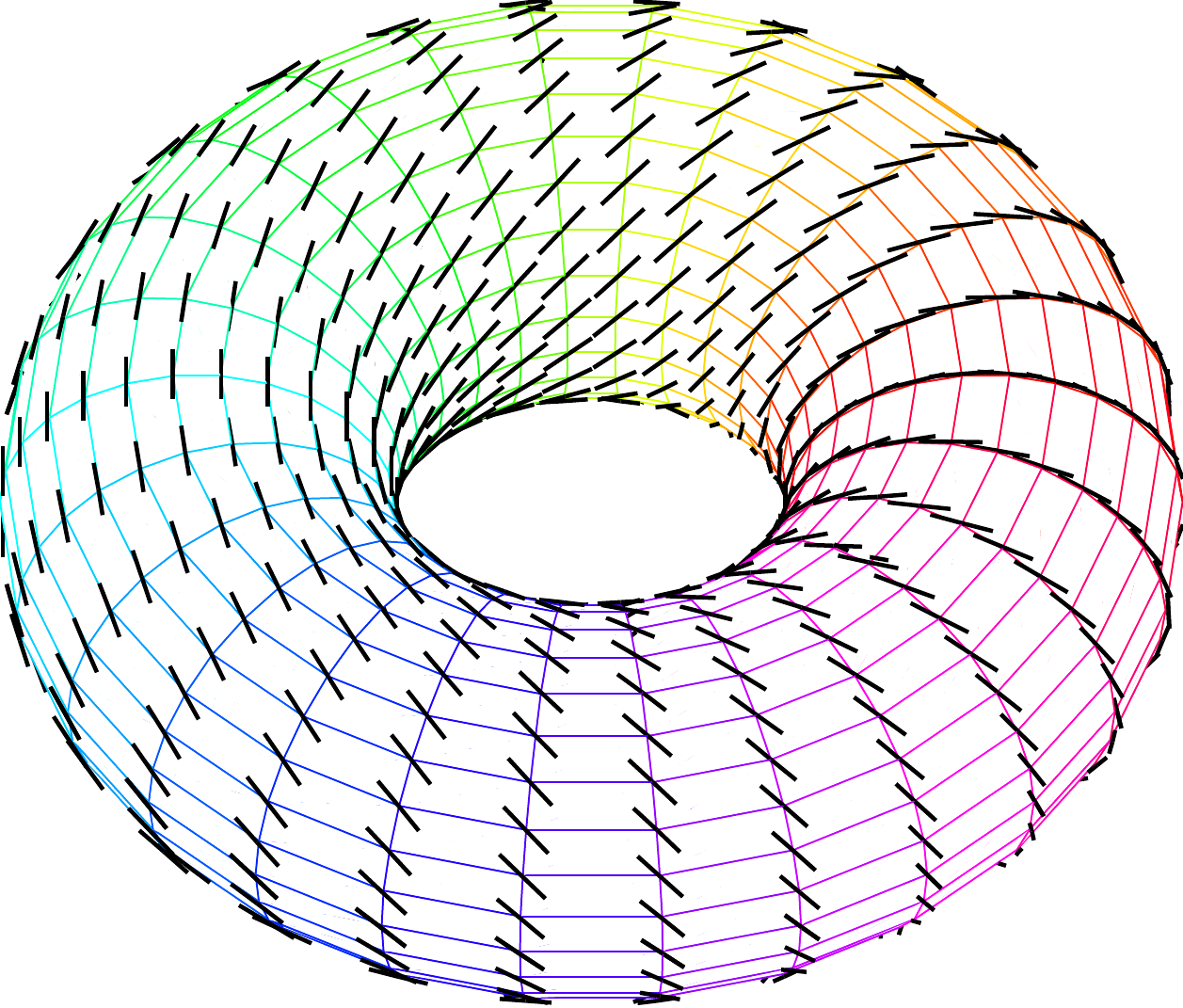} \\ \vspace{5pt}
	\mbox{$n_0^{\otimes 2}$} 
}
	\qquad 
\begin{minipage}{.46\textwidth}
	\centering
	\includegraphics[height=3.5cm]{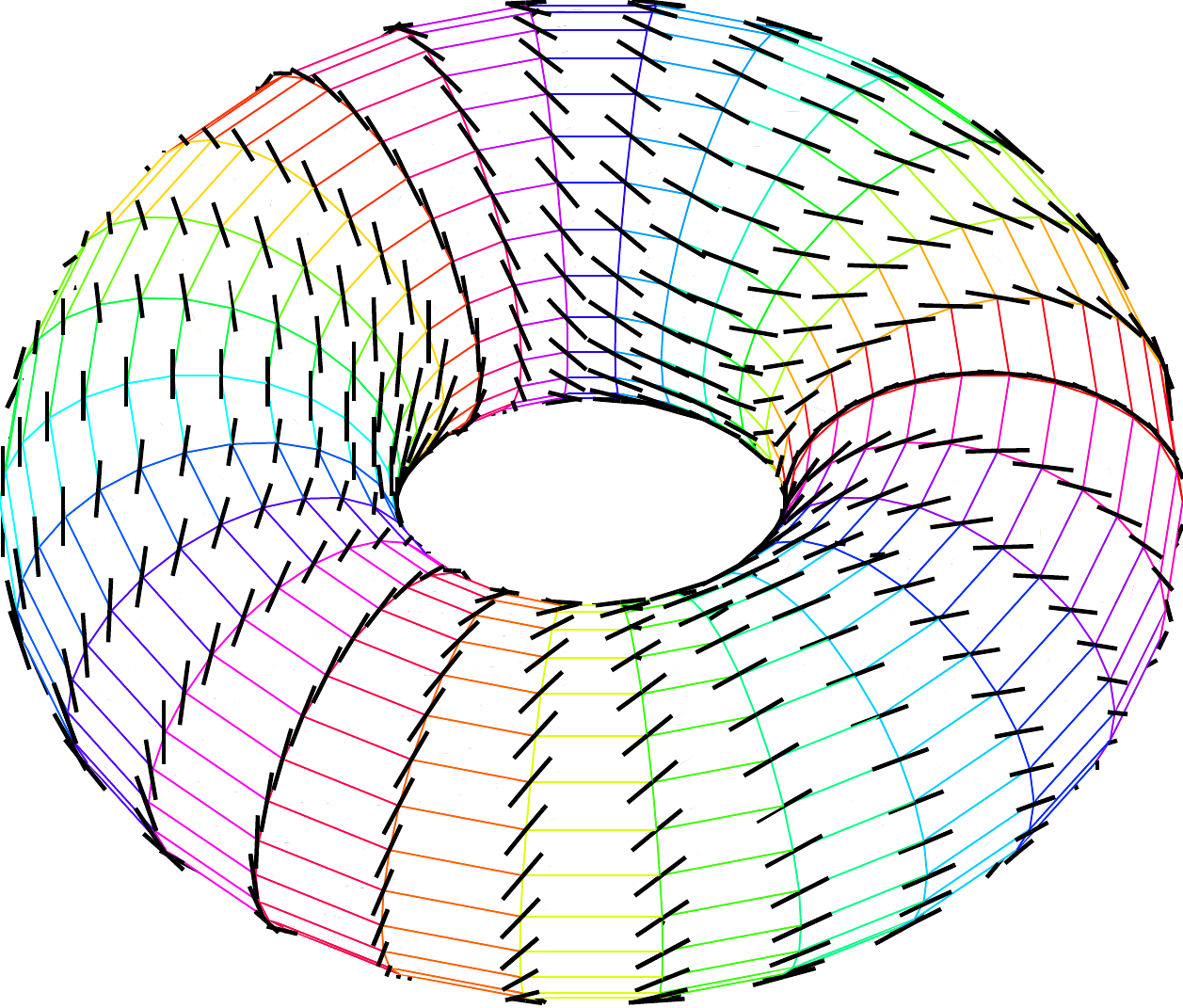} \\ \vspace{5pt}
	\mbox{$n_1^{\otimes 2}$}
\end{minipage}
		\caption{The case of an axisymmetric torus, with radii $R=2,r=1$, parametrized by $X(\theta,\phi) =  	((2+\cos \theta)\cos \phi, 
			(2+ \cos \theta)\sin \phi, 
			\sin \theta)$, on $[0,2\pi[ \times [0,2\pi[$. Let $e_\theta:=\partial_\theta X/|\partial_\theta X|$, $e_\phi:=\partial_\phi X/|\partial_\phi X|$. 
			We give a schematic representation of the two line fields defined via
\[
	n_i(\theta,\phi):= \cos\left(\frac{2i+1}{2}\phi\right)e_\theta +\sin\left(\frac{2i+1}{2}\phi\right)e_\phi,\quad i=0,1.
\]
}
\label{fig_one}
\end{figure}

 \section*{Acknowledgments}
The authors are grateful to Francesco Bonsante, Fabrice Bethuel and Jean Van Schaftingen, for inspiring conversations and suggestions.
We also acknowledge the anonymous referee for his/her careful reading of the manuscript and for his/her
precise comments which surely improved the presentation of the results.
A.S. gratefully acknowledges the Isaac Newton Institute for Mathematics in Cambridge where part of this work has been done during his participation to the program 
\textit{Free Boundary Problems and Related Topics}. Finally, A.S. and M.V. have been supported by the Gruppo Nazionale per l'Analisi Matematica, la Probabilit\`a e le loro Applicazioni (GNAMPA) 
of the Istituto Nazionale di Alta Matematica (INdAM).

 \appendix
 

\section{The index of a continuous field}
\label{app: cont index}

 In this Appendix we collect some known results (see, e.g.,~\cite{hirsch}) on the index for continuous vector fields without assuming the finiteness of its zeros. 
 This goal can be achieved quite straightforwardly, by applying a fundamental tool of differential geometry: the transversality theorem. 
 Such a construction is usually given for granted but, for the reader's convenience, in this section we present its main features. 
 As a consequence of the transversality theorem, we are able to extend some properties of the classical index of a vector field, namely excision, invariance under homotopy, 
 and stability, to continuous vector fields with any number of zeros.
 In Propositions~\ref{prop: cont excision} and~\ref{prop: general homotopy} and in Corollary~\ref{cor: stability} we give the corresponding statements.
 
 Let us start by recalling the definition of transversality. 
 Throughout this section, we denote by $X\subset\R^d$ a compact, connected and oriented manifold without boundary.
 Also, let $E$ be a smooth manifold (without boundary), $\varphi\colon X \to E$ a map of class $\mathscr C^1$, and $Y \subset E$ a submanifold.
 \begin{defn} \label{def: transverse}
  The map $\varphi$ is said to be transverse to $Y$ if and only if, for all $x\in \varphi^{-1}(Y)$, we have 
 \[
  \d \varphi_{x}(T_x X) + T_{\varphi(x)} Y = T_{\varphi(x)} E .
 \]
 \end{defn}
 In our case of interest, $E = TX$ is the tangent bundle of $X$, equipped with the natural projection $\pi\colon E \to X$ given by $(x, \, w)\mapsto x$.
 We take $\varphi$ to be a section of $\pi$ --- that is, a map $\varphi\colon X \to E$ such that $\pi\circ \varphi = \Id_X$.
There is a natural bijection between sections of~$\pi$ and vector fields, i.e. maps~$v\colon X\to\R^d$ which satisfy~$v(x)\in T_x N$ for any~$x\in X$.
 For each section~$\varphi$ can be written in the form
 \begin{equation*} \label{phi-v}
   \varphi(x) = (x, \, v(x)) \qquad \textrm{for all } x\in X
 \end{equation*}
 for a unique vector field~$v$, which is as regular as~$\varphi$. Conversely, given~$v$ this formula uniquely defines a section~$\varphi$ of~$\pi$.
 Finally, we take $Y$ as the image of the zero section, that is,
 \[
  Y := \left\{(x, \, 0) \colon x\in X \right\} \subset E .
 \]
 Clearly, $Y$ is a submanifold of $E$, diffeomorphic to $X$, and $\varphi(x)\in Y$ if and only if $v(x) = 0$.
 
 Fix a point~$x\in X$ and consider a chart $f \colon V \to \R^n$ defined in an open neighbourhood~$V$ of~$x$.
 The map~$f$ naturally induces a chart~$F\colon TV \to \R^{2n}$ of~$TX$, by setting~$F(y, \, v) := (f(y), \, \d f_y(v))$ for any~$y\in V$ and~$v\in T_y X$.
 Let~$f_* v\colon f(V) \subset \R^n \to \R^n$ be defined as in Equation~\eqref{v-coordinates}. Then, there holds
 \[
  \left( F\circ \varphi\circ f^{-1}\right) (z) = \left( z, \, f_*v(z) \right) \qquad \textrm{for } z\in f(V)\subset\R^n
 \]
 and, by interpreting Definition~\ref{def: transverse} through the chart~$F$, we deduce the
 \begin{prop}
 The map $\varphi$ is transverse to $Y$ if and only if for all $x\in v^{-1}(0)$ the differential $\d (f_*v)_{f(x)}$ is invertible.
 \end{prop}
 
 Vector fields in these conditions will simply be called transverse fields.
 As noted in Subsection~\ref{subsect: cont index}, if~$U\subset X$ is an open set and~$v$ is a transverse field on~$\overline U$ such that
 \[
  0 \notin v(\partial U) ,
 \]
 then the index of~$v$ on~$U$ is well-defined by the formula
 \begin{equation} \label{transv index-app}
  \ind(v, \, U) := \sum_{x\in v^{-1}(0)\cap U} \mathrm{sign } \det \d (f_* v)_{f(x)}
 \end{equation}
 or, equivalently, by Formula~\eqref{smooth index}.
 Since we want to extend the definition of index to any continuous field, it is natural to ask whether a continuous field can be approximated by transverse fields.
 The transversality theorem gives a positive answer.
 This result, due to Thom (see \cite{Thom1, Thom2}), states that transverse mappings are a dense subset of continuous mappings. 
 The statement that we present here is \cite[Theorem 14.6]{BJ}.
 This formulation is convenient for our purposes, because it guarantees that if $\varphi$ is a section of $\pi$, then the approximating transverse maps can be chosen to be sections as well.
 
 \begin{theor}[Transversality theorem] \label{th: transversality}
 Let $\pi\colon E \to X$ be a smooth vector bundle, $Y$ a submanifold of $E$, and $\varphi\colon X \to E$ a smooth section of $\pi$.
 Then, given any continuous function $\varepsilon \colon X \to (0, \, +\infty)$, there exists a section $\psi$ of $\pi$ which is transverse to $Y$ and satisfies
 \[
  \norm{\varphi(x) - \psi(x)}_{T_x E} \leq \varepsilon(x)  \qquad \textrm{for all } x\in X .
 \]
 Moreover, if $A\subset X$ is a closed set such that $\varphi_{|A}$ is of class $\mathscr C^1$ and transverse to $Y$, 
 then one can choose $\psi$ so that $\psi_{|A} = \varphi_{|A}$.
 \end{theor}
 
 The smoothness assumption on $\varphi$ is not really a restriction, because every continuous section can be approximated with smooth sections (e.g., working in coordinate charts which trivialize $\pi$).
 Hence, from this theorem we immediately obtain the result we need about vector fields.
 
 \begin{cor} \label{cor: approx field} 
  Let $U$ be an open subset of $X$, and let $v$ be a continuous vector field defined on~$\overline U$. 
  If $v$ satisfies $0 \notin v(\partial U)$, then there exists a transverse field $u$ on $\overline U$, such that
  \begin{align} 
   &  u \text{ has finitely many zeros}, \label{approx finite}\\ 
   & \sup_{x\in \overline U}\abs{v(x) - u(x)} < \inf_{x\in \partial U}\abs{v(x)}.\label{approx field}
  \end{align}
 \end{cor}
 
 Now we can define the index of an arbitrary field.
 
 \begin{defn} \label{def: cont index}
  Let $v$ be a continuous vector field on $U$, such that $0\notin v(\partial U)$. If $v$ is transverse, we define $\ind(v, \, U)$ by formula~\eqref{transv index}.
  Otherwise, we define 
  \[
   \ind(v, \, U) := \ind(u, \, U),
  \]
  where $u$ is any transverse field satisfying~\eqref{approx field}.
 \end{defn}
 
The well-posedness of this definition follows from the homotopy invariance of the index for transverse vector fields, and can be proved by arguing exactly as for Corollary~\ref{cor: stability}.

 The definition of index closely resembles Brouwer's construction of the degree. This similarity is not coincidental.
 Indeed, as we mentioned in the Introduction, an equivalent way of making sense of the index for an arbitrary continuous field is to define it as the degree of an appropriate map.
 
 \begin{remark} \label{remark: tubular}
  More precisely, consider a tubular neighbourhood $M\subset \R^d$ of the manifold $X$, i.e., an open neighbourhood of $X$ in $\R^d$ such that any point
  $y\in M$ can be uniquely decomposed as $y = x + \nu$, where $x\in X$ and $\nu$ is orthogonal to $T_x X$.
  Let $\tau\colon M \to X$ be the map given by $y \mapsto x$, which is smooth if $M$ is small enough. 
  Consider the normal extension of $v$, that is, the continuous function $w\colon M \to \R^d$ given by
  \[
  w (y) := v(\tau(y)) + y - \tau(y) \qquad \textrm{for all } y\in M .
  \]
  Then, we can set
  \begin{equation} \label{ind-deg}
  \ind(v, \, U):= \deg(w, \, \tau^{-1}(U), \, 0) .
  \end{equation}
  It is not hard to see that this quantity coincides with the index in the sense of Definition~\ref{def: cont index}.
  Actually, by means of Brezis and Nirenberg degree theory, the right-hand side in this formula makes sense when $v$ is just $\VMO$ (and satisfies a suitable nonvanishing condition near the boundary).
  Thus, one could consider taking~\eqref{ind-deg} as a general definition of index.
  However, for a $\VMO$ field $v$ this approach does not allow to define the quantity $\ind_-(P_{\partial N}v, \, \partial_- N[v])$, which occurs in Morse's formula, because $\partial_-N [v]$ may not be open.
  Henceforth, one would still have to consider continuous fields at first, then take care of the $\VMO$ case by an approximation procedure.
 \end{remark}
 
Due to this strong link between the index and the degree, it is not surprising that some important properties of the degree have a counterpart for the index.
 We collect them in the next Propositions, 
 leaving the proofs to the reader.
 The first property we consider here is excision.

  \begin{prop}[Excision] \label{prop: cont excision}
  Let $U_1 \subset U$, $U_2 \subset U$ be two disjoint open sets in $X$, and let $v$ be a continuous vector field on $X$. If $0 \notin v(\overline U \setminus (U_1\cup U_2))$, then
  \[
   \ind(v, \, U) = \ind(v, \, U_1) + \ind(v, \, U_2).
  \]
 \end{prop}
 
 The second property is the invariance of the index under a continuous homotopy. 
 We state a first version of this principle, in which we allow both the vector field and the underlying domain to vary continuously. 
 The proof is analogous to~\cite{Samelson}.
 
 \begin{prop}[General homotopy principle] \label{prop: general homotopy}
 Let $\{M_t\}_{0 \leq t \leq 1}$ be a family of compact, oriented $n$-manifolds in $\R^d$, without boundary, such that the set
 \[
  M := \coprod_{0 \leq t \leq 1} M_t \times \{t\}
 \]
 is a $(n + 1)$-submanifold of $\R^d\times[0, \, 1]$.
 Let $V$ be an open,  connected subset of $M$, and set $V_t := V \cap (\R^d \times \{ t \})$.
 Let $v\colon \overline V \to \R^d$ be a continuous map such that, for each $0 \leq t \leq 1$,
 \begin{itemize}
  \item[(i)] $v(\cdot, \, t)$ is a tangent field to $M_t$, and
  \item[(ii)] $0 \notin v(\partial V_t)$.
 \end{itemize}
 Then, for any $0 \leq t_1 , \, t_2 \leq 1$ such that $V_{t_1}\neq \emptyset$, $V_{t_2}\neq \emptyset$, we have
 \[
  \ind(v(\cdot, \, t_1), \, V_{t_1}) = \ind(v(\cdot, \, t_2), \, V_{t_2}) .
 \]
 \end{prop}

 In case the domain is fixed, from this general principle we can derive the stability of the index with respect to small perturbations of the fields. 
 
 \begin{cor}[Stability] \label{cor: stability}
 Let $v_0, \, v_1$ be two continuous vector fields on $\overline U$, satisfying $0\notin v_0(\partial U)$, $0 \notin v_1(\partial U)$. If
 \begin{equation} \label{smallpert}
   \abs{v_0(x) - v_1(x)} < \abs{v_0(x)} \qquad \textrm{for all } x\in \partial U ,
 \end{equation}
 then $\ind(v_0, \, U) = \ind(v_1, \, U)$.
 \end{cor}
 

An important consequence of Corollary~\ref{cor: stability} is that all the continuous vector fields have the same index on $X$.
This agrees with the Poincar\'e-Hopf formula, which yields $\ind(v, \, X) = \chi(X)$.
 
\medskip
In view of the previous discussion, we can give a meaning to the quantities defined in Subsection~\ref{subsect: cont index} for any continuous field~$v$ satisfying~$0 \notin v(\partial N)$.
To define the index of~$v$, we apply Definition~\ref{def: cont index}, taking as~$X$ the topological double of~$N$ and~$U = N\setminus\partial N$. 
To define the inward boundary index, we apply Definition~\ref{def: cont index} to the vector field~$P_{\partial N}v$ and take~$X = \partial N$, $U = \partial_- N[v]$.

We conclude our discussion by giving the proofs of Propositions~\ref{prop: stability} and~\ref{prop: continuous index formula}.

\begin{proof}[Proof of Proposition~\ref{prop: stability}]
Let $v$, $w$ be two continuous fields on~$N$, satisfying~$0\notin v(\partial N)$, $0\notin w(\partial N)$ and
\begin{equation} \label{inw stability}
  \norm{v - w}_{\mathscr C(\partial N)} < \varepsilon_1 : = \frac{\sqrt 5 -1}{4}\min_{\partial N}{|v|}.
\end{equation}
First of all, note that since  $\varepsilon_1 < \min_{\partial N}|v|$, Corollary~\ref{cor: stability} applies and so~$\ind(v, \, N) = \ind(w, \, N)$. 
Therefore, it only remains to prove that the boundary indices of~$v$ and~$w$ agree.

For the sake of simplicity, set~$c := \min_{\partial N}|v| >0$. Due to~\eqref{inw stability}, we deduce
 \begin{equation} \label{inw stability 1}
  \abs{\frac{v(x) \cdot \nu(x)}{\abs{v(x)}} - \frac{w(x)\cdot \nu(x)}{\abs{w(x)}}} \leq \frac{2\varepsilon_1}{c -  \varepsilon_1} \qquad \textrm{for all } x\in \partial N .
 \end{equation}
 Indeed, for a fixed $x\in \partial N$ we suppose, e.g.,  that $\abs{w(x)} \leq \abs{v(x)}$. Then
\[
\begin{split}
  \abs{\frac{v(x) \cdot \nu(x)}{\abs{v(x)}} - \frac{w(x)\cdot \nu(x)}{\abs{w(x)}}} 
  &\leq \abs{\frac{v(x) \cdot \nu(x)}{\abs{v(x)}} - \frac{v(x)\cdot \nu(x)}{\abs{w(x)}}} + \abs{\frac{v(x) \cdot \nu(x)}{\abs{w(x)}} - \frac{w(x)\cdot \nu(x)}{\abs{w(x)}}} \\
  &\leq \abs{v(x)}\left(\frac{1}{\abs{w(x)}} - \frac{1}{\abs{v(x)}} \right) + \frac{\abs{v(x) - w(x) }}{\abs{w(x)}}   \\
  &=\frac{\abs{v(x)} - \abs{w(x)}}{\abs{w(x)}} + \frac{\abs{v(x) - w(x) }}{\abs{w(x)}} \\
  &\leq 2 \frac{\abs{v(x) - w(x) }}{\abs{w(x)}} ,
\end{split}
 \]
 whence the desired inequality~\eqref{inw stability 1}. Thus, setting
 \[
   U_+ := \left\{ x\in \partial N \colon \frac{w(x)\cdot\nu(x)}{\abs{w(x)}} < \frac{2\varepsilon_1}{c -  \varepsilon_1} \right\}
 \]
 and
 \[
   U_- := \left\{ x\in \partial N \colon \frac{w(x)\cdot\nu(x)}{\abs{w(x)}} < -\frac{2\varepsilon_1}{c -  \varepsilon_1} \right\} ,
 \]
from~\eqref{inw stability 1} it follows that 
 \[
   U_- \subset \partial_- N[v] \subset U_+ \qquad \textrm{and} \qquad \partial(\partial_- N[v]) \subset \overline{U_+} \setminus U_- .
 \]
 Moreover, for all $x\in \overline{U_+} \setminus U_-$ the conditions~\eqref{inw stability} and~\eqref{inw stability 1} imply
 \begin{equation} \label{inw stability 2}
  \abs{P_{\partial N} w(x)} \geq \abs{w(x)} \sqrt{1 - \frac{4\varepsilon_1^2}{(c - \varepsilon_1)^2}} \geq \sqrt{(c - \varepsilon_1)^2 - 4\varepsilon_1^2}.
 \end{equation}
 By definition, $\varepsilon_1$ is a solution to
 \[
  \varepsilon_1 = \sqrt{(c - \varepsilon_1)^2 - 4\varepsilon_1^2} 
 \]
 and so, in $\overline{ U_+}\setminus U_{-}$ there holds 
 \[
	 |P_{\partial N}v - P_{\partial N}w|\leq |v-w| \stackrel{\eqref{inw stability}}{<} \sqrt{(c - \varepsilon_1)^2 - 4\varepsilon_1^2} \stackrel{\eqref{inw stability 2}}{\leq} |P_{\partial N}w|. 
 \]
The condition~\eqref{smallpert} is thus satisfied, so that we can apply Corollary~\ref{cor: stability} to $P_{\partial N}v, \, P_{\partial N}w$, to infer
 \[
   \ind_- (v, \, \partial N) = \ind(P_{\partial N}v, \, \partial_- N[v]) = \ind(P_{\partial N}w, \, \partial_- N[v]) .
 \]
On the other hand, by~\eqref{inw stability 2} there is no zero of $P_{\partial N} w$ in the region $\overline{U_+}\setminus U_-$, 
which contains the symmetric difference between $\partial_- N[v]$ and $\partial_- N[w]$. 
Hence, Proposition~\ref{prop: cont excision} gives
 \[
  \ind(P_{\partial N}w, \, \partial_- N[v]) = \ind(P_{\partial N}w, \, \partial_- N[w]) = \ind_- (w, \, \partial N).
 \]
This concludes the proof.
\end{proof}
 
 We can now prove Proposition~\ref{prop: continuous index formula}.
 This result can be obtained combining the classical Morse's index formula (Proposition~\ref{prop: morse index}) with~\cite[$\Gamma$-existence Theorem]{Pugh}, 
 which provides a transverse vector field~$v$ such that~$P_{\partial N} v$ is also transverse.
 However, for the convenience of the reader, we present here a proof which does not rely on the results in~\cite{Pugh}.
%
 \begin{proof}[Proof of Proposition~\ref{prop: continuous index formula}]
 We show that it is possible to approximate both~$v$ and~$P_{\partial N} v$ using the same transverse field~$u$. 
 Then, the proposition will follow by applying the classical Morse's formula to~$u$.

 Owning to the continuity of~$v$, we find a number~$c > 0$ and a neighbourhood~$U$ of~$\partial N$ in~$N$ such that
\begin{equation} \label{morse 1}
 \abs{v(x)} \geq c \qquad \textrm{for all } x\in U.
\end{equation}
Let~$\varepsilon > 0$ be a small parameter, to be chosen later.
We fix a smooth vector field~$\tilde v$ on~$N$ such that
\begin{equation} \label{morse 5}
 \norm{v - \tilde v}_{\mathscr C(N)} \leq \varepsilon .
\end{equation}
Then, by Theorem~\ref{th: transversality}, we approximate~$P_{\partial N}\tilde v$ with a transverse vector field~$\xi$ on~$\partial N$, such that~$\xi$ has finitely many zeros on~$\partial N$ and
\begin{equation} \label{morse 6}
 \norm{P_{\partial N}\tilde v - \xi}_{\mathscr C(\partial N)} \leq \varepsilon.
\end{equation}
We claim that there exists a continuous vector field~$w$ on~$N$, which is smooth on~$U$, satisfies
\[
  w = \begin{cases}
       \xi + \tilde v - P_{\partial N} \tilde v \qquad &\textrm{on } \partial N \\
       \tilde v                                     \qquad &\textrm{on } N \setminus U
      \end{cases}
\]
and 
\begin{equation} \label{morse 3}
 \norm{v - w}_{\mathscr C(N)} \leq C\varepsilon ,
\end{equation}
for some constant~$C$ depending only on~$N$.
(Remark that the prescribed boundary value for~$w$ is compatible with the condition~\eqref{morse 3}, as it follows from~\eqref{morse 5} and~\eqref{morse 6}).
We are giving the details of this construction in a moment, but first, we show how to conclude the proof.

By construction,~$w_{|\partial N}$ is a smooth function and 
$w(x)\in T_x N$ for any~$x\in N$. For~$\varepsilon$ small enough,~\eqref{morse 3} and 
Proposition~\ref{prop: stability} entail that
\begin{equation} \label{morse 2}
 \ind_-(v, \, \partial N) = \ind_- (w, \, \partial N) .
\end{equation}

Take~$\varepsilon < c/C$. Then,~\eqref{morse 1} and~\eqref{morse 3} together imply that~$w$ does not vanish on~$U$.
In particular,~$w$ is vacuously transverse on~$U$.
Using Theorem~\ref{th: transversality}, we modify~$w$ out of~$U$ to get a transverse vector field~$u$, such that~$u_{|U} = w_{|U}$.
As~$u$ can be taken arbitrarily close to~$w$ in the~$\mathscr C$-norm, we can assume that~\eqref{approx field} is satisfied.
Hence,
\begin{equation} \label{morse 4}
 \ind(v, \, N) = \ind(u, \,  N) .
\end{equation}
Since~$u$ is a transverse field, with finitely many zeros, Morse's identity applies to~$u$. 
Then, using~\eqref{morse 2} and~\eqref{morse 4}, the proposition follows.

 Now, let us explain how to construct the map~$w$.
 Taking a smaller~$U$ if necessary, we can assume that~$U$ is a collar of~$\partial N$.
 This means,~$U$ is of the form
 \[
  U = \left\{x\in N \colon \dist(x, \, \partial N) \leq \delta \right\}
 \]
 for some~$\delta > 0$, each point~$x\in U$ has a unique nearest projection~$\sigma(x)\in \partial N$, and the mapping~$\varphi$ given by
 \[
  \varphi(x) := (\sigma(x), \abs{x - \sigma(x)}) \qquad \textrm{for } x\in U
 \]
 is a diffeormorphism~$U \to \partial N \times [0, \, \delta]$.
 For each~$x\in U$, the differential~$d\varphi_x$ is an isomorphism
 \[
  T_x N \simeq T_{\sigma(x)} \partial N \oplus \R ,
 \]
 so~$T_xN$ can be decomposed into a tangential and a normal subspace, with respect to~$\partial N$.
 To keep the notation simple, we assume here that~$U = \partial N \times [0, \, \delta]$, and~$\varphi = \Id_U$.
 
 To define~$w$, we interpolate linearly between~$\xi$ and the tangential component of~$\tilde v$, but we leave the normal component of~$\tilde v$ unchanged.
 More precisely, given~$x = (y, \, t)\in \partial N \times [0, \, \delta]$ we define 
\begin{align*}
  w(x) &:= \left(1 - \frac{t}{\delta}\right)\xi(y) + \frac{t}{\delta} P_{\partial N} \tilde v(x) + \tilde v(x) - P_{\partial N} \tilde v(x)\\
       &\,= \left(1 - \frac{t}{\delta}\right)\Big(\xi(y) - P_{\partial N} \tilde v(x)\Big) + \tilde v(x) 
 \end{align*}
 whereas we set
 \[
  w(x) := \tilde v(x) \qquad \textrm{for } x\in N \setminus U . 
 \]
 Then~$w$ is of class~$\mathscr C^1$ on~$U$, continuous on~$N$, satisfies~$w = \xi + \tilde v - P_{\partial N} \tilde v$ on~$\partial N$.
 Moreover, for~$x = (y, \, t)\in U$ we have
 \[
 \begin{split}
  \abs{\tilde v(x) - w(x)} & \leq \left(1 - \frac{t}{\delta}\right)\big|\xi(y) -  P_{\partial N} \tilde v(x)\big|\\
  		&\leq \left(1 - \frac{t}{\delta}\right) \left(\big|\xi(y) -  P_{\partial N} \tilde v(y,0)\big| 
			+ \big| P_{\partial N} \tilde v(y,0)-P_{\partial N} \tilde v(y,t)\big|\right)\\
		&\stackrel{\eqref{morse 6}}{\leq} \left(1 - \frac{t}{\delta}\right)\varepsilon +t\left(1 - \frac{t}{\delta}\right)
			 \text{Lip}_U(P_{\partial N} \tilde v)\\
		&\stackrel{\eqref{morse 6}}{\leq} \varepsilon +\delta C .
  \end{split}
 \]
 By choosing~$\delta$ small, and combining this inequality with~\eqref{morse 5}, we deduce~\eqref{morse 3}.
\end{proof}

\bibliographystyle{acm}
\bibliography{Index}

\end{document}